\documentclass[12pt]{amsart}
\usepackage{amsfonts}
\usepackage{bbm}
\usepackage{mathrsfs}
\usepackage{txfonts}
\usepackage{amssymb}
\usepackage{graphicx, epsf,amssymb,amscd,amsfonts,times}
\usepackage{color,xspace,colortbl}
\usepackage{mathtools}
\usepackage{overpic}
\usepackage{latexsym}

\newtheorem{thm}{Theorem}[section]

\newtheorem{prop}[thm]{Proposition}
\newtheorem{defn}[thm]{Definition}
\newtheorem{lemma}[thm]{Lemma}
\newtheorem{cor}[thm]{Corollary}

\newtheorem{expl}[thm]{Example}

\theoremstyle{remark}
\newtheorem{rmk}[thm]{Remark}

\newenvironment{notation}
{\noindent\ignorespaces}
{\par\noindent\ignorespacesafterend}

\newcommand{\bdry}{\partial}

\newcommand{\be}{\begin{enumerate}}
\newcommand{\ee}{\end{enumerate}}

\begin{document}
%%%%%%%%%%%%%%%%%%%%%%%%%%%%%%%%%%%%%%%

\title[A proof of the classification theorem of overtwisted contact structures]{A proof of the classification theorem of overtwisted contact structures via convex surface theory}

\author{Yang Huang}
\address{University of Southern California, Los Angeles, CA 90089}
\email{huangyan@usc.edu}

\begin{abstract}
In~\cite{El}, Y. Eliashberg proved that two overtwisted contact structures on a closed oriented 3-manifold are isotopic through contact structures if and only if they are homotopic as 2-plane fields. We provide an alternative proof of this theorem using the convex surface theory and bypasses.
\end{abstract}

\maketitle
%%%%%%%%%%%%%%%%%%%%%%%%%%%%%%%%%%%%%%%

\tableofcontents

%%%%%%%%%%%%%%%%%%%%%%%%%%%%%%%%%%%%%%%

A contact manifold $(M,\xi)$ is a smooth manifold with a contact structure $\xi$, i.e., a maximally non-integrable codimension 1 tangent distribution. In particular, if the dimension of the manifold is three, it was realized through the work of D. Bennequin and Y. Eliashberg in~\cite{Be},~\cite{El1} that contact structures fall into two classes: {\em tight} or {\em overtwisted}. Since then, dynamical systems and foliation theory of surfaces embedded in contact 3-manifolds have been studied extensively to analyze this dichotomy. Based on these developments, Eliashberg gave a classification of overtwisted contact structures in~\cite{El}, which we now explain.

Let $M$ be a closed oriented manifold and $\triangle\subset M$ be an oriented embedded disk. Furthermore, we fix a point $p\in\triangle$. We denote by $Cont^{ot}(M,\triangle)$ the space of cooriented, positive, overtwisted contact structures on $M$ which are overtwisted along $\triangle$, i.e., the contact distribution is tangent to $\triangle$ along $\bdry\triangle$. Let $Distr(M,\triangle)$ be the space of cooriented 2-plane distributions on $M$ which are tangent to $\triangle$ at $p$. Both spaces are equipped with the $C^\infty$-topology.

\begin{thm}[Eliashberg]
Let $M$ be a closed, oriented 3-manifold. Then the inclusion $j:Cont^{ot}(M,\triangle) \to Distr(M,\triangle)$ is a homotopy equivalence.
\end{thm}

In particular, we have:

\begin{thm} \label{OT}
Let $M$ be a closed, oriented 3-manifold. If $\xi$ and $\xi'$ are two positive overtwisted contact structures on $M$, then they are isotopic if and only if they are homotopic as 2-plane fields.
\end{thm}

Consequently, overtwisted contact structures are completely determined by the homotopy classes of the underlying 2-plane fields. On the other hand, the classification of tight contact structures is much more subtle and contains more topological information about the ambient 3-manifold.

The goal of this paper is to provide an alternative proof of Theorem~\ref{OT} based on convex surface theory. Convex surface theory was introduced by E. Giroux in~\cite{Gi} building on the work of Eliashberg-Gromov~\cite{ElGr}. Given a closed oriented surface $\Sigma$, we consider contact structures on $\Sigma\times[0,1]$ such that $\Sigma\times\{0,1\}$ is convex. By studying the ``film picture'' of the {\em characteristic foliations} on $\Sigma\times\{t\}$ as $t$ goes from 0 to 1, Giroux showed in~\cite{Gi1} that, up to an isotopy, there are only finitely many levels $\Sigma\times\{t_i\}$, $0<t_1<\cdots<t_n<1$, which are not convex. Moreover, for small $\epsilon>0$, the characteristic foliations on $\Sigma\times\{t_i-\epsilon\}$ and $\Sigma\times\{t_i+\epsilon\}$, $i=1,2,\cdots,n$, change by a {\em bifurcation}. In~\cite{Ho}, K. Honda gave an alternative description of the bifurcation of characteristic foliations in terms of {\em dividing sets}. Namely, he defined an operation, called the {\em bypass attachment}, which combinatorially acts on the dividing set. It turns out that a bypass attachment is equivalent to a bifurcation on the level of characteristic foliations. Hence, in order to study contact structures on $\Sigma \times [0,1]$ with convex boundary, it suffices to consider the isotopy classes of contact structures given by sequences of bypass attachments. In particular, we will study sequences of (overtwisted) bypass attachments on $S^2 \times [0,1]$, which is the main ingredient in our proof of Theorem~\ref{OT}.

This paper is organized as follows. In Section 1 we recall some basic knowledge in contact geometry, in particular, convex surface theory and the definition of a bypass. Section 2 gives an outline of our approach to the classification problem. Section 3 is devoted to establishing some necessary local properties of the bypass attachment. Using techniques from previous sections, we show in Section 4 that how to isotop homotopic overtwisted contact structures so that they agree in a neighborhood of the 2-skeleton. Section 5, 6 and 7 are devoted to studying overtwisted contact structures on $S^2\times[0,1]$ which is the technical part of this paper. We finally finish the proof of Theorem~\ref{OT} in Section 8.

\section{Preliminaries}

Let $M$ be a closed, oriented 3-manifold. Throughout this paper, we only consider cooriented, positive contact structures $\xi$ on $M$, i.e., those that satisfy the following conditions:

\be

\item{there exists a global 1-form $\alpha$ such that $\xi=\ker(\alpha)$.}

\item{$\alpha \wedge d\alpha>0$, i.e., the orientation induced by the contact form $\alpha$ agrees with the orientation on $M$.}

\ee

A contact structure $\xi$ is {\em overtwisted} if there exists an embedded disk $D^2 \subset M$ such that $\xi$ is tangent to $D^2$ on $\bdry D^2$. Otherwise, $\xi$ is said to be {\em tight}. We will focus on overtwisted contact structures for the rest of this paper.

Let $\Sigma \subset M$ be a closed, embedded, oriented surface in $M$. The {\em characteristic foliation} $\Sigma_\xi$ on $\Sigma$ is by definition the integral of the singular line field $\Sigma_\xi(x)\coloneqq \xi_x \cap T_x\Sigma$. One way to describe the contact structure near $\Sigma$ is to look at its characteristic foliation.

\begin{prop}[Giroux]
Let $\xi_0$ and $\xi_1$ be two contact structures which induce the same characteristic foliation on $\Sigma$. Then there exists an isotopy $\phi_t:M \to M$, $t\in[0,1]$ fixing $\Sigma$ such that $\phi_0=id$ and $(\phi_1)_*\xi_0=\xi_1$.
\end{prop}

Possibly after a $C^{\infty}$-small perturbation, we can always assume that $\Sigma\subset M$ is {\em convex}, i.e., there exists a vector field $v$ transverse to $\Sigma$ such that the flow of $v$ preserves the contact structure. Using this transverse contact vector field $v$, we define the {\em dividing set} on $\Sigma$ to be $\Gamma_\Sigma\coloneqq \{x\in\Sigma~|~v_x \in \xi_x\}$. Note that the isotopy class of $\Gamma_\Sigma$ does not depend on the choice of $v$. We refer to~\cite{Gi} for a more detailed treatment of basic properties of convex surfaces. The significance of dividing sets in contact geometry is made clear by {\em Giroux's flexibility theorem}.

\begin{thm}[Giroux] \label{Flex}
Assume $\Sigma$ is convex with characteristic foliation $\Sigma_\xi$, contact vector field $v$, and dividing set $\Gamma_\Sigma$. Let $\mathscr{F}$ be another singular foliation on $\Sigma$ divided by $\Gamma_\Sigma$. Then there exists an isotopy $\phi_t: M \to M, t\in [0,1]$ such that

\be

\item{$\phi_0=id$ and $\phi_t|_{\Gamma_\Sigma}=id$ for all $t$.}

\item{$v$ is transverse to $\phi_t(\Sigma)$ for all $t$.}

\item{$\phi_1(\Sigma)$ has characteristic foliation $\mathscr{F}$.}

\ee
\end{thm}

We now look at contact structures on $\Sigma\times[0,1]$ with convex boundary. The first important result relating to this problem is the following theorem due to Giroux.

\begin{thm}[Giroux] \label{film pic}
Let $\xi$ be a contact structure on $W=\Sigma\times[0,1]$ so that $\Sigma\times\{0,1\}$ is convex. There exists an isotopy relative to the boundary $\phi_s:W \to W$, $s\in[0,1]$, such that the surfaces $\phi_1(\Sigma\times\{t\})$ are convex for all but finitely many $t\in[0,1]$ where the characteristic foliations satisfy the following properties:
\be
\item{The singularities and closed orbits are all non-degenerate.}
\item{The limit set of any half-orbit is either a singularity or a closed orbit.}
\item{There exists a single ``retrogradient'' saddle-saddle connection, i.e., an orbit from a negative hyperbolic point to a positive hyperbolic point.}
\ee
\end{thm}

In the light of Giroux's flexibility theorem, one should expect a corresponding ``film picture'' of dividing sets on convex surfaces. It turns out that the correct notion corresponding to a bifurcation is the {\em bypass attachment}, which we now describe.

\begin{defn}\label{bypass}
Let $\Sigma$ be a convex surface and $\alpha$ be a Legendrian arc in $\Sigma$ which intersects $\Gamma_\Sigma$ in three points, two of which are endpoints of $\alpha$. A {\em bypass} is a convex half-disk $D$ with Legendrian boundary, where $D \cap \Sigma=\alpha$, $D\pitchfork\Sigma$, and $tb(\bdry D)=-1$. We call $\alpha$ an admissible arc, and $D$ a bypass along $\alpha$ on $\Sigma$.
\end{defn}

\begin{rmk}
The admissible arc $\alpha$ in the above definition is also known as the {\em arc of attachment} for a bypass in literature.
\end{rmk}

\begin{rmk}
We do not distinguish isotopic admissible arcs $\alpha_0$ and $\alpha_1$, i.e., if there exists a path of admissible arcs $\alpha_t$, $t\in[0,1]$ connecting them.
\end{rmk}

The following lemma shows how a bypass attachment combinatorially acts on the dividing set.

\begin{lemma}[Honda]
Following the terminology from Definition~\ref{bypass}, let $D$ be a bypass along $\alpha$ on $\Sigma$. There exists a neighborhood of $\Sigma\cup D\subset M$ diffeomorphic to $\Sigma\times[0,1]$, such that $\Sigma\times\{0,1\}$ are convex, and $\Gamma_{\Sigma\times\{1\}}$ is obtained from $\Gamma_{\Sigma\times\{0\}}$ by performing the {\em bypass attachment} operation as depicted in Figure~\ref{BypassAttach} in a neighborhood of $\alpha$.

\end{lemma}

\begin{figure}[h]
  \includegraphics[width=0.35\textwidth]{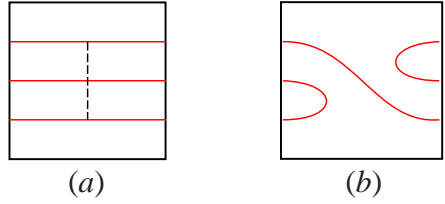}
  \put(-142,-12){$(a)$}
  \put(-37,-12){$(b)$}
  \caption{A bypass attachment along $\alpha$. (a) The dividing set on $\Sigma\times\{0\}$ before the bypass is attached. (b) The dividing set on $\Sigma\times\{1\}$ after the bypass is attached.}
  \label{BypassAttach}
\end{figure}

It is worthwhile to mention that there are two distinguished bypasses, namely, the trivial bypass and the overtwisted bypass as depicted in Figure~\ref{TrivialOT}. The effect of a trivial bypass attachment is isotopic to an $I$-invariant contact structure where no bypass is attached, while the overtwisted bypass attachment immediately introduces an overtwisted disk in the local model, hence, for example, is disallowed in tight contact manifolds.
\begin{figure}[h]
  \begin{overpic}[scale=.4]{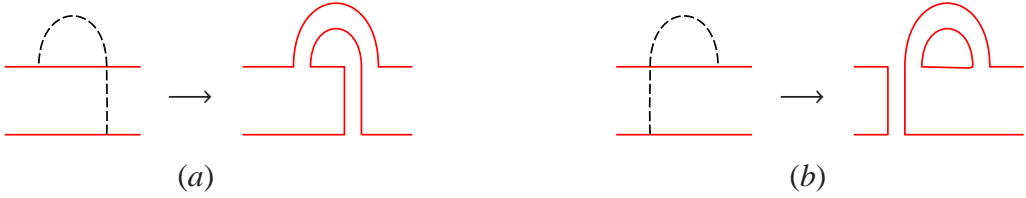}
  \put(16,3){$\longrightarrow$}
  \put(76,3){$\longrightarrow$}
  \put(17,-5){$(a)$}
  \put(77,-5){$(b)$}
  \end{overpic}
  \newline
  \caption{(a) The trivial bypass attachment. (b) The overtwisted bypass attachment.}
  \label{TrivialOT}
\end{figure}

\section{Outline of the proof}

Let $\xi$ and $\xi'$ be two overtwisted contact structures on $M$, homotopic to each other as 2-plane field distributions. Our approach to Theorem~\ref{OT} has three main steps.\\

\noindent\textit{Step 1.} Fix a triangulation $T$ of $M$. Isotop $\xi$ and $\xi'$ through contact structures such that $T$ becomes an {\em overtwisted contact triangulation} in the sense that the 1-skeleton $T^{(1)}$ is a Legendrian graph, the 2-skeleton $T^{(2)}$ is convex and each 3-cell is an overtwisted ball with respect to both contact structures. We first show that if $e(\xi)=e(\xi') \in H^2(M;\mathbb{Z})$, then one can isotop $\xi$ and $\xi'$ so that they agree in a neighborhood of $T^{(2)}$.\\

\noindent\textit{Step 2.} We can assume that there exists a ball $B^3 \subset M$ such that $\xi$ and $\xi'$ agree on $M\setminus B^3$. Taking a small Darboux ball $B^3_{std} \subset B^3$, observe that $\xi|_{B^3}$ and $\xi'|_{B^3}$ can both be realized as attaching sequences of bypasses to $B^3_{std}$. In section 5, we will define the notion of a {\em stable isotopy}. Then we show that both of sequences of bypass are stably isotopic to some power of the {\em bypass triangle attachment}. Moreover, the boundary relative homotopy classes of $\xi|_{B^3}$ and $\xi'|_{B^3}$, measured by the Hopf invariant, are uniquely determined by the number of bypass triangles attached according to~\cite{Hu}.\\

\noindent\textit{Step 3.} By elementary obstruction theory, the Hopf invariants of $\xi|_{B^3}$ and $\xi'|_{B^3}$ are not necessarily the same, but they can at most differ by an integral multiple of the divisibility of the Euler class of either $\xi$ or $\xi'$. See Section 8 for the definition of the divisibility. We show that this ambiguity can be resolved by further isotoping the contact structures in a neighborhood of $T^{(2)}$. This finishes the proof of Theorem~\ref{OT}.

\section{Local properties of bypass attachments}

Let $M$ be an overtwisted contact 3-manifold. Let $\Sigma \subset M$ be a closed convex surface with dividing set $\Gamma_\Sigma$. For convenience, we choose a metric on $M$ and denote $M\setminus\Sigma$ the metric closure of the open manifold $M-\Sigma$. In this paper, we restrict ourself to the case that each connected component of $M \setminus \Sigma$ is overtwisted\footnote{In general it is possible that all components of $M \setminus \Sigma$ are tight even if $M$ is overtwisted.}. In order to isotop convex surfaces through bypasses freely, we must show that there are enough bypasses. In fact, bypasses exist along any admissible Legendrian arc on $\Sigma$ provided that the contact structure is overtwisted. This is the content of the following lemma.

\begin{lemma} \label{Abund}
Suppose that $M\setminus\Sigma$ is overtwisted. For any admissible arc $\alpha\subset\Sigma$, there exists a bypass along $\alpha$ in $M\setminus\Sigma$. If $\Sigma$ separates $M$ into two overtwisted components, then there exists such a bypass in each component.
\end{lemma}

\begin{proof}
The technique for proving this lemma is essentially due to Etnyre and Honda~\cite{EH}, and independently Torisu~\cite{To}. We construct a bypass $D$ along $\alpha$ as follows. Let $\tilde{D} \subset M \setminus \Sigma$ be an overtwisted disk.

First we push the interior of $\alpha$ slightly into $M\setminus\Sigma$ with the endpoints of $\alpha$ fixed to obtain another Legendrian arc $\tilde\alpha$, such that $\alpha$ and $\tilde\alpha$ cobound a convex bigon $B$ with $tb(\bdry B)=-2$. Next, take a Legendrian arc $\gamma$ connecting $\tilde{\alpha}$ and $\bdry\tilde{D}$ in the complement of $\Sigma \cup \tilde D \cup B$, namely, the two endpoints of $\gamma$ are contained in $\tilde{\alpha}$ and $\bdry\tilde{D}$ respectively and the interior of $\gamma$ is disjoint from $\Sigma \cup \tilde{D} \cup B$ as depicted in Figure~\ref{legarc}.
\begin{figure}[h]
    \centering
    \begin{overpic}[scale=.33]{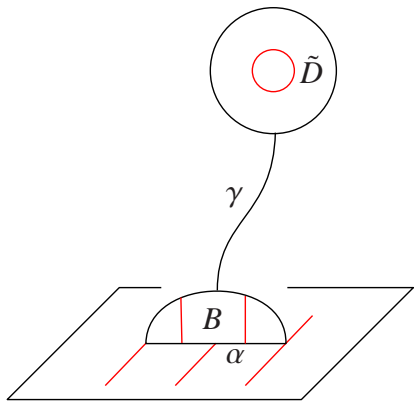}
    \put(48,17.5){$B$}
    \put(72,78){$\tilde D$}
    \put(54,49){$\gamma$}
    \put(54,9){$\alpha$}
    \end{overpic}
    \caption{The Legendrian arc $\gamma$ connecting $\bdry B$ and $\bdry \tilde D$.}
    \label{legarc}
\end{figure}
Suppose $N(\gamma) \cong \gamma\times[-\epsilon,\epsilon]$ is a band with the core $\gamma\times\{0\}$ identified with $\gamma$, such that the characteristic foliation is non-singular and is given by $\gamma\times\{t\}$, $t \in [-\epsilon,\epsilon]$. In particular $\gamma\times\{-\epsilon\}$ and $\gamma\times\{\epsilon\}$ are both Legendrian. We want to glue $N(\gamma)$ to $\tilde D$ and $B$ so that the characteristic foliations match along the common boundary. In order to do so, we recall the following lemma first observed by Fraser~\cite{Fr}.

\begin{lemma} \label{pivot}
Let $S$ be an embedded disk in a contact manifold $(M, \xi)$ with a characteristic foliation $\xi|_S$ which consists only of one positive
elliptic singularity p and unstable orbits from p which exit transversely from $\bdry S$. If $\delta_1$, $\delta_2$ are two unstable orbits meeting at $p$, and $\delta_i \cap S = p_i$, then, after a $C^\infty$-small perturbation of $S$ fixing $\bdry S$, we obtain $S'$ whose characteristic
foliation has exactly one positive elliptic singularity $p'$ and unstable orbits from $p'$ exiting transversely from $\bdry S$, and for which the orbits passing through $p_1$, $p_2$ meet tangentially at $p'$.
\end{lemma}

We first glue $N(\gamma)$ to $\tilde D$ as follows. Let $p_1 = \gamma \cap \bdry \tilde D$. By the Flexibility Theorem we may suppose that $p_1$ is a half-elliptic singular point of the characteristic foliation $\xi|_{\tilde D}$ on $\tilde D$. Consider a slightly larger disk $\tilde D' \supset \tilde D$ such that $p_1$ is an elliptic singularity of $\xi|_{\tilde D'}$. Let $S \subset \tilde D'$ be a small disk neighbothood of $p_1$, which satisfies the conditions in Lemma~\ref{pivot}. Applying Lemma~\ref{pivot}, we can perturb $S$ to get a disk $\hat D$ on which the characteristic foliation (in a neighbothood of $p_1$) looks like the one depicted in Figure~\ref{neck}. 

\begin{figure}[h]
    \centering
    \begin{overpic}[scale=.5]{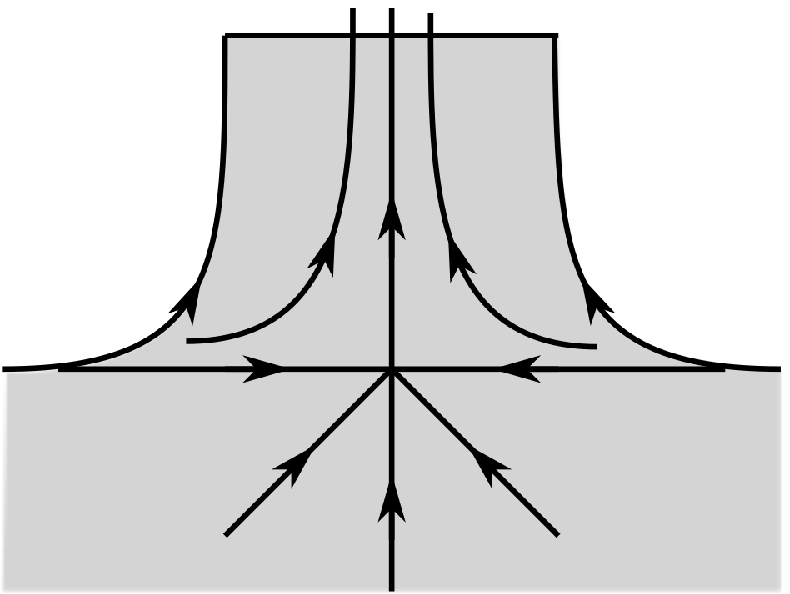}
    \put(51.5,32.5){\small{$p_1$}}
    \put(83,8){\small{$\hat D$}}
    \end{overpic}
    \caption{}
    \label{neck}
\end{figure}

Now we can glue $N(\gamma)$ to $\hat D$ in the obvious way such that the characteristic foliations match along the common boundary. We can apply the same trick to glue $N(\gamma)$ to $B$. In the end we obtain a half disk, which we denote by $\tilde D \cup N(\gamma) \cup B$ by abuse of notation, on which the characteristic foliation is as depicted in Figure~\ref{charbypass}.

\begin{figure}[h]
    \centering
    \begin{overpic}[scale=.5]{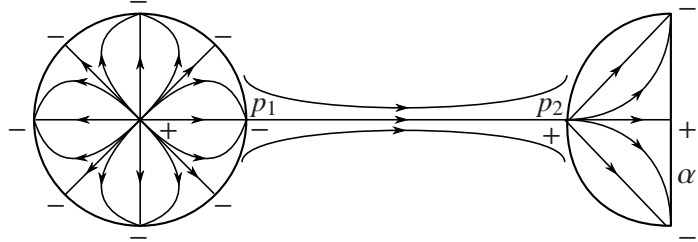}
    \put(101,7){$\alpha$}
    \put(34,18.2){\small{$p_1$}}
    \put(78.8,18.2){\small{$p_2$}}
    \put(19.7,13.8){$+$}
    \put(34,14){$-$}
    \put(79.7,13){$+$}
    \put(101,14){$+$}
    \put(101,-3){$-$}
    \put(101,33){$-$}
    \put(-4,14){$-$}
    \put(2,2){$-$}
    \put(2,28.5){$-$}
    \put(15,-3){$-$}
    \put(15,34){$-$}
    \put(28.2,2){$-$}
    \put(28.2,28.5){$-$}
    \end{overpic}
    \vspace{3mm}
    \caption{The preferred characteristic foliation on $\tilde D \cup N(\gamma) \cup B$.}
    \label{charbypass}
\end{figure}

Note that since the characteristic foliation contains a flowline from the negative half-elliptic-half-hyperbolic singularity to the positive half-elliptic-half-hyperbolic singularity, the half disk $\tilde D \cup N(\gamma) \cup B$ is not convex. However we can perform a $C^\infty$-small perturbation in a neighborhood of $p_1$ and $p_2$ to obtain a new half disk $D$ such that the singularities $p_1$ and $p_2$ are eliminated. The characteristic foliation on $D$ is given by Figure~\ref{charbypassfinal}, which is easily seen to be of Morse-Smale type. Therefore $D$ is convex with Legendrian boundary. The dividing set $\Gamma$ on $D$ has to separate the positive and negative singularities and to be transverse to the characteristic foliation. So $\Gamma$ is, up to isotopy, the half-circle as depicted in Figure~\ref{charbypassfinal} as desired, and therefore $D$ is a bypass along $\alpha$.
\begin{figure}[h]
    \centering
    \vspace{3mm}
    \begin{overpic}[scale=.5]{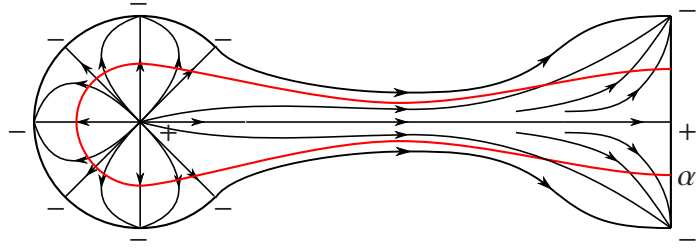}
    \put(101,7){$\alpha$}
    \put(19.7,13.8){$+$}
    \put(101,14){$+$}
    \put(101,-3){$-$}
    \put(101,33){$-$}
    \put(-4,14){$-$}
    \put(2,2){$-$}
    \put(2,28.5){$-$}
    \put(15,-3){$-$}
    \put(15,34){$-$}
    \put(28.2,2){$-$}
    \put(28.2,28.5){$-$}
    \end{overpic}
    \vspace{3mm}
    \caption{The bypass $D$ along $\alpha$.}
    \label{charbypassfinal}
\end{figure}
\end{proof}

We then show the triviality of the trivial bypass, i.e., attaching a trivial bypass does not change the isotopy class of the contact structure in a neighborhood of the convex surface. The proof essentially follows the lines of the proof of Proposition 4.9.7 in Geiges~\cite{Ge}. Here the contact structure may be either overtwisted or tight.

\begin{lemma} \label{Triviality}
Let $(\Sigma\times[0,1],\xi)$ be a contact manifold with the contact structure $\xi$ obtained by attaching a trivial bypass on $(\Sigma\times\{0\},\xi|_{\Sigma\times\{0\}})$. Then there exists another contact structure $\tilde\xi$, which is isotopic to $\xi$ relative to the boundary, such that $\Sigma\times\{t\}$ is convex with respect to $\tilde\xi$ for all $t\in[0,1]$.
\end{lemma}

\begin{proof}
Since this is a local problem, we may assume that $\Sigma\times[0,1]$ is a neighborhood of the trivial bypass attachment. By Theorem~\ref{Flex}, any Morse-Smale type characteristic foliation adapted to $\Gamma_{\Sigma\times\{0\}}$ can be realized as the characteristic foliation of a contact structure isotopic to $\xi$ in a neighborhood of $\Sigma\times\{0\}$. In particular, we can assume that the characteristic foliation on $\Sigma\times\{0\}$ looks exactly the same as in Figure~\ref{Bif}(a) such that $e_-$ does not connect to any negative hyperbolic point other than $h_-$ along the flow line.

\begin{figure}[h]
    \begin{overpic}[scale=.3]{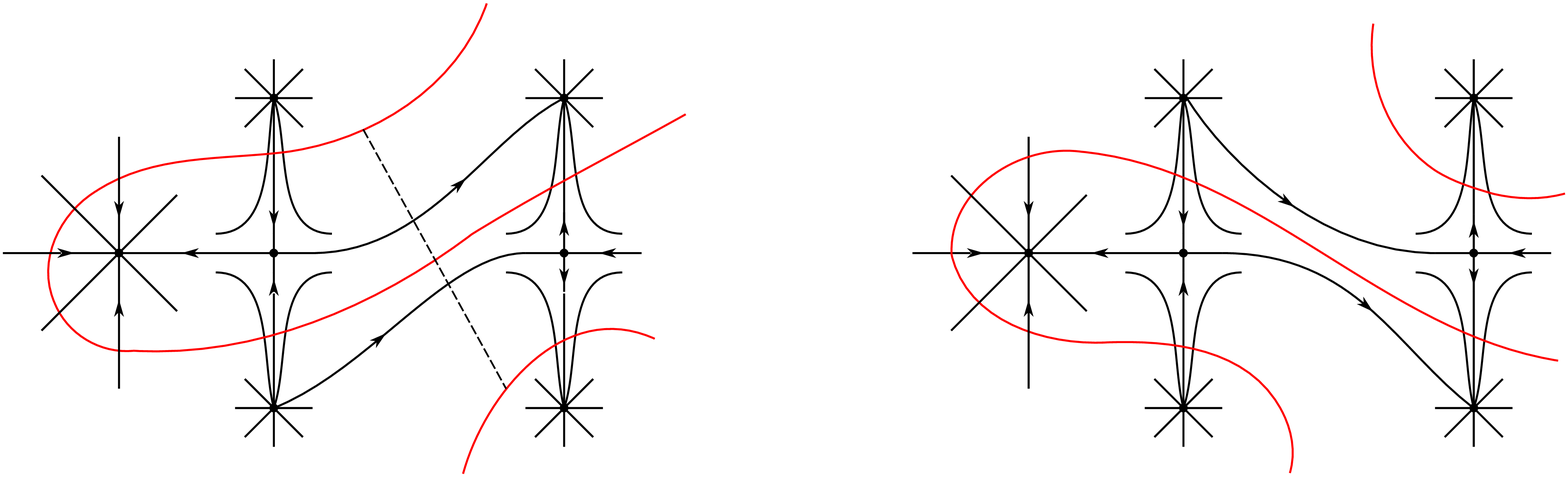}
    \put(4.4,12.9){\tiny{$e_-$}}
    \put(15.4,12.7){\tiny{$h_-$}}
    \put(33.7,12.7){\tiny{$h_+$}}
    \put(62.5,12.9){\tiny{$e_-$}}
    \put(73.3,12.7){\tiny{$h_-$}}
    \put(91.8,12.7){\tiny{$h_+$}}
    \put(20,-4){(a)}
    \put(80,-4){(b)}
    \end{overpic}
    \newline
    \caption{(a) The characteristic foliation on $\Sigma\times\{0\}$. The trivial bypass is attached along the Legendrian arc in dash line. (b) The characteristic foliation on $\Sigma\times\{1\}$ after attaching the trivial bypass. Here $e_{\pm}$ (resp. $h_{\pm}$) denote the $\pm$-elliptic (resp. $\pm$-hyperbolic) singular points of the foliation.}
    \label{Bif}
\end{figure}

Look at the characteristic foliations on $\Sigma\times\{t\}$ as $t$ goes from 0 to 1. Generically we can assume that the Morse-Smale condition fails at one single level, say, $\Sigma\times\{1/2\}$, where an unstable saddle-saddle connection has to appear as shown in Figure~\ref{SScon}(a).

\begin{figure}[h]
    \begin{overpic}[scale=.35]{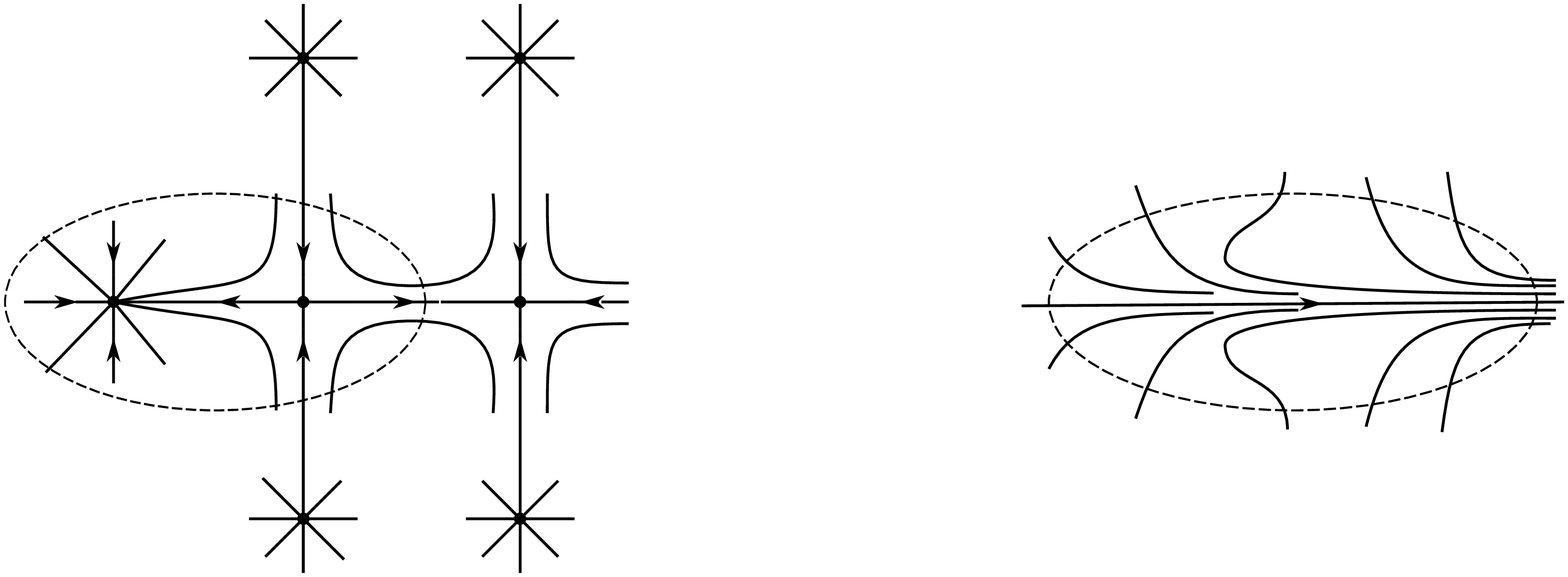}
    \put(3.4,15.8){\tiny{$e_-$}}
    \put(17,15.5){\tiny{$h_-$}}
    \put(30.5,15.5){\tiny{$h_+$}}
    \put(-1,23){$\Omega$}
    \put(66,23){$\Omega$}
    \put(20,-5){(a)}
    \put(80,-5){(b)}
    \end{overpic}
    \newline
    \caption{(a) The characteristic foliation on $\Sigma\times\{1/2\}$, where a saddle-saddle connect from $h_-$ to $h_+$ exists. The region $\Omega$ contains exactly two singular points $\{e_-,h_-\}$ which are in elimination position. (b) The nonsingular characteristic foliation on $\Omega$ after the elimination.}
    \label{SScon}
\end{figure}

Let $\Omega \subset \Sigma\times\{1/2\}$ be an open neighborhood of the flow line from $h_-$ to $e_-$ as depicted in Figure~\ref{SScon}(a). Observe that the characteristic foliation inside $\Omega$ is of Morse-Smale type, and therefore stable in the $t$-direction. According to the proof of Proposition 4.9.7\footnote{This is a stronger version of the usual Elimination Lemma.} in Geiges~\cite{Ge}, for a small $\delta>0$, there exists an isotopy $\phi_s:\Sigma\times[0,1] \to \Sigma\times[0,1]$, $s \in [0,1]$, compactly supported in $\Omega\times(1/2-2\delta,1/2+2\delta) \subset \Sigma\times[0,1]$ and $\phi_0=id$, such that $\tilde\xi=(\phi_1)_*\xi$ satisfies the following:

\be

\item{The characteristic foliation on $\Omega\times\{t\}$ with respect to $\tilde\xi$ is isotopic to the one in Figure~\ref{SScon}(b) for $t\in[1/2-\delta,1/2+\delta]$. In particular, it is nonsingular.}

\item{For $t\in(1/2-2\delta,1/2-\delta)\cup(1/2+\delta,1/2+2\delta)$, The characteristic foliation on $\Omega\times\{t\}$ with respect to $\tilde\xi$ is almost Morse-Smale except that there may exist a half-elliptic-half-hyperbolic point.}

\ee
We remark here that the above conditions are achieved in~\cite{Ge} by isotoping surfaces $\Sigma\times\{t\}$, $t\in[1/2-2\delta,1/2+2\delta]$ while fixing the contact structure $\xi$, but this is equivalent to isotoping $\xi$ while fixing $\Sigma\times\{t\}$. We will switch between these two equivalent point of view again in the proof of Proposition~\ref{2skeleton}.

Now we can make $\Sigma\times\{t\}$ convex for $t\in[1/2-\delta,1/2+\delta]$ because the only unstable saddle-saddle connection is eliminated and therefore the characteristic foliation becomes Morse-Smale. For $t\notin[1/2-\delta,1/2+\delta]$, there may exist half-elliptic-half-hyperbolic singular points, but we can as well construct a contact structure realizing this type of singularity so that each $\Omega\times\{t\}$ stays convex. Hence $\tilde\xi$ constructed above is as required.
\end{proof}

\begin{rmk}
Let $(\Sigma\times[0,1],\xi)$ be a contact manifold such that $\xi|_{\Sigma_0}=\xi|_{\Sigma_1}$ and $\Sigma\times\{t\}$ is convex for all $t\in[0,1]$. If $\Sigma\neq S^1\times S^1$ and $\xi$ is tight, then it is a standard fact that $\xi$ is isotopic to an $I$-invariant contact structure relative to the boundary. However, if either $\Sigma=S^1\times S^1$ or $\xi$ is overtwisted, then the above fact is not true anymore. We will study this phenomenon in detail in the case when $\Sigma=S^2$ and $\xi$ is overtwisted in Section 6.
\end{rmk}

\section{Isotoping contact structures up to the 2-skeleton}

We are now ready to take the first main step towards the proof of Theorem~\ref{OT}. Since we will isotop contact structures skeleton by skeleton, we start with the following definition.

\begin{defn}
Let $(M,\xi)$ be an overtwisted contact manifold, and $T$ be a triangulation of $M$. The triangulation $T$ is called an {\em overtwisted contact triangulation} if the following conditions hold:
\be
\item{The 1-skeleton is a Legendrian graph.}
\item{Each 2-simplex is convex with Legendrian boundary.}
\item{Each 3-simplex is an overtwisted ball.}
\ee
\end{defn}

\begin{rmk}
The overtwisted contact triangulation defined above is different from the usual {\em contact triangulation} where the 3-simplexes are assumed to be tight.
\end{rmk}

The goal for this section is to prove the following Proposition.

\begin{prop} \label{2skeleton}
Let $M$ be a closed, oriented 3-manifold with a fixed triangulation $T$. Let $\xi$ and $\xi'$ be homotopic overtwisted contact structures on $M$. Then they are isotopic up to the 2-skeleton, i.e., there exists an isotopy $\phi_t:M \to M$, $t \in [0,1]$, $\phi_0=id$ such that $(\phi_1)_*\xi=\xi'$ in a neighborhood of $T^{(2)}$.
\end{prop}

\begin{proof}
Before we go into details of the proof, observe that if $\phi_t:M \to M$, $t \in [0,1]$, $\phi_0=id$ is an isotopy, then $(M,\phi_1(\xi),T)$ and $(M,\xi,\phi_1^{-1}(T))$ carries the same contact information. In fact, we will isotop the skeletons of the triangulation $T$ and think of them as isotopies of contact structures.

By a $C^0$-small perturbation of the 1-skeleton $T^{(1)}$, we can assume that $T^{(1)}$ is a Legendrian graph with respect to $\xi$ and $\xi'$. Performing stabilizations to edges of $T^{(1)}$ if necessary, we can further assume that $\xi=\xi'$ in a neighborhood of $T^{(1)}$. For each 2-simplex $\sigma^2$ in $T^{(2)}$, we can always stabilize the Legendrian unknot $\bdry \sigma^2$ sufficiently many times so that $tb(\bdry \sigma^2)<0$. Therefore a $C^\infty$-small perturbation of $\sigma^2$ relative to $\bdry\sigma^2$ makes it convex with respect to $\xi$ (resp. ${\xi'}$) with dividing set $\Gamma_{\sigma^2}^\xi$ (resp. $\Gamma_{\sigma^2}^{\xi'}$). Both $\Gamma_{\sigma^2}^\xi$ and $\Gamma_{\sigma^2}^{\xi'}$ are proper 1-submanifolds of $\sigma^2$ and generically the endpoints are contained in the interior of the 1-simplexes. See Figure~\ref{DivSet} for an example.

In order to make $T$ an overtwisted contact triangulation for $\xi$ and $\xi'$, we still need to make sure that all 3-simplexes are overtwisted. We do this for $\xi$, and the same argument applies to $\xi'$. Take an overtwisted disc $D$ in $(M,\xi)$. We can assume that $D$ is contained in a 3-simplex $\sigma^3_1$. Let $\sigma^3_2$ be another 3-simplex which shares a 2-face with $\sigma^3_1$, i.e., $\sigma^3_1 \cap \sigma^3_2=\sigma^2$ is a 2-simplex. We claim that by isotoping $\sigma^2$ relative to $\bdry \sigma^2$ if necessary, we can make both $\sigma^3_1$ and $\sigma^3_2$ overtwisted. The fact that $M$ is closed immediately implies that a finite steps of such isotopies will make $T$ an overtwisted contact triangulation. To prove the claim, we first take a parallel copy of the overtwisted disk $D$ in an $I$-invariant neighborhood of $D$, denoted by $D'$. Pick an arc $\gamma$ connecting $D'$ to $\sigma^2$ inside $\sigma^3_1$. Let $\tilde\sigma^2$ be another 2-simplex obtained by isotoping $\sigma^2$ across $D'$ along $\gamma$, i.e., $\tilde\sigma^2$ satisfying the following conditions:

\be
\item{$\bdry\tilde\sigma^2=\bdry\sigma^2$.}
\item{$\sigma^2\cup\tilde\sigma^2$ bounds a neighborhood of $D'\cup\gamma$.}
\item{$\tilde\sigma^2$ is convex.}
\ee

By replacing $\sigma^2$ with $\tilde\sigma^2$, we obtain two new 3-simplexes, each of which contains an overtwisted disk in the interior as claimed.

\begin{figure}[h]
    \begin{overpic}[scale=.2]{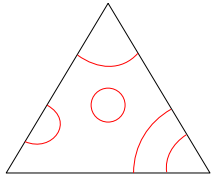}
    \end{overpic}
    \caption{An example of the dividing set on a 2-simplex.}
    \label{DivSet}
\end{figure}

Now by Giroux's flexibility theorem, it suffices to isotop $\xi$ and ${\xi'}$ so that they induce isotopic dividing sets on each 2-simplex relative to $T^{(1)}$. To achieve this goal, we define the difference 2-cocycle $\delta$ by assigning to each oriented 2-simplex $\sigma^2$ an integer $\chi(R_+(\Gamma^{\xi'}_{\sigma^2}))-\chi(R_-(\Gamma^{\xi'}_{\sigma^2}))-\chi(R_+(\Gamma^\xi_{\sigma^2}))+\chi(R_-(\Gamma^\xi_{\sigma^2}))$. Since $\xi$ is homotopic to ${\xi'}$ as 2-plane fields, $[\delta]=e(\xi)-e(\xi')=0\in H^2(M,\mathbb{Z})$. Hence there exists an integral 1-cocycle $\theta$ so that $2d\theta=\delta$ since the Euler class is always even.\footnote{More precisely, if we fix a trivialization of $TM$ and consider the Gauss map associated to the contact distribution, then the Euler class of the contact distribution is exactly twice the Poincar\'e dual of the Pontryagin submanifold of the Gauss map.} One should think of $\theta$ as an element in $Hom(C_1(M),\mathbb{Z})$.

Let $\sigma^2 \in T^{(2)}$ be an oriented convex 2-simplex and $\sigma^1 \subset \bdry\sigma^2$ be an oriented 1-simplex with the induced orientation. We study the effect of stabilizing the 1-simplex $\sigma^1$ to the overtwisted contact triangulation. If we positively stabilize $\sigma^1$ once and isotop $\sigma^2$ accordingly to obtain a new 2-simplex $\tilde\sigma^2$, then the dividing set $\Gamma^{\xi}_{\tilde\sigma^2}$ on $\tilde\sigma^2$ is obtained from $\Gamma^{\xi}_{\sigma^2}$ by adding a properly embedded arc contained in the negative region with both endpoints on the interior of $\sigma^1$ as depicted in Figure~\ref{PosStab}. Similarly, if we negatively stabilize $\sigma^1$ once and isotop $\sigma^2$ accordingly as before, then the dividing set on the isotoped $\sigma^2$ is obtained from $\Gamma^{\xi}_{\sigma^2}$ by adding a properly embedded arc contained in the positive region and with both endpoints on the interior of $\sigma^1$.

\begin{figure}[h]
    \begin{overpic}[scale=.3]{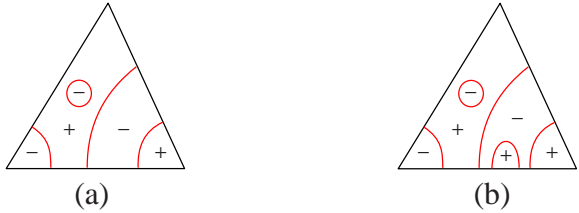}
    \put(3.5,2){\tiny{$-$}}
    \put(10,6){\tiny{$+$}}
    \put(19.5,6){\tiny{$-$}}
    \put(26,2){\tiny{$+$}}
    \put(11.7,12.5){\tiny{$-$}}
    \put(72,2){\tiny{$-$}}
    \put(78,6){\tiny{$+$}}
    \put(88.5,8){\tiny{$-$}}
    \put(86.5,1.5){\tiny{$+$}}
    \put(94.6,2){\tiny{$+$}}
    \put(80.3,12.5){\tiny{$-$}}
    \put(12,-6){(a)}
    \put(82,-6){(b)}
    \end{overpic}
    \newline
    \caption{(a) The dividing set on $\sigma^2$ divides it into $\pm$-regions. The bottom edge is $\sigma^1$. (b) One possible dividing set on $\tilde\sigma^2$ after positively stabilizing $\sigma^1$ once.}
    \label{PosStab}
\end{figure}

Note that in general, the new overtwisted contact triangulation obtained by $\pm$-stabilizing a 1-simplex $\sigma^1$ is not unique. In fact, different choices may give non-isotopic dividing sets on the isotoped $\sigma^2$ in the new triangulation. However, for our purpose, we only care about the quantity $\chi(R_+)-\chi(R_-)$ on each 2-simplex and it is easy to see that different choices give the same value to this quantity. Thus we will ignore this ambiguity by arbitrarily choosing an isotopy of the 2-simplex.

We denote the overtwisted contact triangulation obtained by $\pm$-stabilizing $\sigma^1$ once in $(M,\xi)$ by $S^{\pm}_{\sigma^1}(\xi)$. As remarked at the beginning of the proof, one should think of $S^{\pm}_{\sigma^1}(\xi)$ as isotopies of $\xi$. It is easy to see that $S^{\pm}_{\sigma^1}(\xi)$ changes $\chi(R_+(\Gamma^\xi_{\sigma^2}))-\chi(R_-(\Gamma^\xi_{\sigma^2}))$ by $\pm 1$ for any 2-simplex $\sigma^2 \in T^{(2)}$ so that $\sigma^1 \subset \bdry\sigma^2$ as an oriented boundary edge. The same holds for $\xi'$ as well.

Now we argue that one can isotop $\xi$ and ${\xi'}$ so that $\chi(R_+(\Gamma^\xi_{\sigma^2}))-\chi(R_-(\Gamma^\xi_{\sigma^2}))=\chi(R_+(\Gamma^{\xi'}_{\sigma^2}))-\chi(R_-(\Gamma^{\xi'}_{\sigma^2}))$ on each 2-simplex $\sigma^2$. This can be done as follows. For each oriented 1-simplex $\sigma^1 \in T^{(1)}$, the 1-cocycle $\theta$ sends it to an integer $n=\theta(\sigma^1)$. We perform $n$ times the isotopy $S^+_{\sigma^1}(\xi)$ to $\xi$ and $n$ times the isotopy $S^-_{\sigma^1}({\xi'})$ to ${\xi'}$ at the same time. If we perform such operation to every 1-simplex in $T$, it is easy to see that the following properties are satisfied:

\be

\item{$\xi={\xi'}$ in a neighborhood of $T^{(1)}$.}

\item{$\chi(R_+(\Gamma^\xi_{\sigma^2}))-\chi(R_-(\Gamma^\xi_{\sigma^2}))=\chi(R_+(\Gamma^{\xi'}_{\sigma^2}))-\chi(R_-(\Gamma^{\xi'}_{\sigma^2}))$, $\forall \sigma^2 \in T^{(2)}$.}

\ee

The second property implies that $\Gamma^{\xi'}_{\sigma^2}$ can be obtained from $\Gamma^\xi_{\sigma^2}$ by attaching a sequence of bypasses for each 2-simplex $\sigma^2$. Recall that $T$ is an overtwisted contact triangulation and in particular each 3-simplex is an overtwisted ball. Hence bypasses exist along any admissible arc in $\sigma^2$ inside any 3-simplex with $\sigma^2$ as a 2-face by Lemma~\ref{Abund}. Therefore by isotoping 2-simplexes through bypasses, we can assume that $\xi$ and $\xi'$ induce isotopic dividing sets on each 2-simplex relative to its boundary. The conclusion now follows immediately from Giroux's flexibility theorem.
\end{proof}

\section{Bypass triangle attachments}

In this section we study the effect of attaching a bypass triangle to the contact structure, in particular, we give an alternative definition of the bypass triangle attachment. We start with the definition of the bypass triangle attachment.\\

\begin{notation}
{\em Notation:} Let $\Sigma$ be a convex surface and $\alpha \subset \Sigma$ be an admissible arc. We denote the bypass attachment along $\alpha$ on $\Sigma$ by $\sigma_\alpha$. Let $\beta$ be another admissible arc on the convex surface obtained by attaching the bypass along $\alpha$ on $\Sigma$. We denote the composition of bypass attachments by $\sigma_\alpha\ast\sigma_\beta$, where the composition rule is to attach the bypass along $\alpha$ first, then attach the bypass along $\beta$ in the same direction. If $(M,\xi)$ is a contact manifold with convex boundary, then $\xi\ast\sigma_\alpha$ denotes the contact structure obtained by attaching a bypass along $\alpha$ to $(M,\xi)$.
\end{notation}

\begin{rmk}
In general, bypass attachments are not commutative unless the attaching arcs are disjoint.
\end{rmk}

\begin{defn}
Let $\Sigma$ be a convex surface and $\alpha \subset \Sigma$ be an admissible arc. A {\em bypass triangle attachment} along $\alpha$ is the composition of three bypass attachments along admissible arcs $\alpha$, $\alpha'$ and $\alpha''$ in a neighborhood of $\alpha$ as depicted in Figure~\ref{BypassTri}. We denote the bypass triangle attachment along $\alpha$ by $\triangle_\alpha=\sigma_\alpha \ast \sigma_{\alpha'} \ast \sigma_{\alpha''}$.
\end{defn}

\begin{rmk} \label{anibypassarc}
The second admissible arc $\alpha'$ in the bypass bypass triangle is also known as the {\em arc of anti-bypass attachment} to $\sigma_\alpha$.
\end{rmk}

\begin{figure}[h]
  \begin{overpic}[scale=.32]{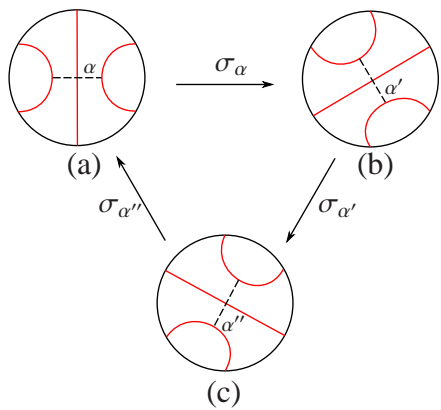}
  \put(13.5,46){(a)}
  \put(81,46){(b)}
  \put(46,-7){(c)}
  \put(17.5,69.5){\tiny{$\alpha$}}
  \put(87.5,64.5){\tiny{$\alpha'$}}
  \put(49.5,9.5){\tiny{$\alpha''$}}
  \put(48,70){\small{$\sigma_\alpha$}}
  \put(72,37){\small{$\sigma_{\alpha'}$}}
  \put(21,37){\small{$\sigma_{\alpha''}$}}
  \end{overpic}
  \newline
  \caption{(a) A neighborhood of $\alpha$ on $\Sigma$, along which the first bypass $\sigma_\alpha$ is attached. (b) The second bypass $\sigma_{\alpha'}$ is attached along the dotted arc $\alpha'$. (c) The third bypass $\sigma_{\alpha''}$ is attached along the dotted arc $\alpha''$ and finishes the bypass triangle.}
  \label{BypassTri}
\end{figure}

\begin{notation}
\textit{Warning}: When we define a bypass attachment $\sigma_\alpha$ along $\alpha$ on $(\Sigma,\Gamma_\Sigma)$, there are several choices involved. Namely, we need to choose a multicurve, i.e., a 1-submanifold of $\Sigma$, representing the isotopy class of $\Gamma_\Sigma$, an admissible arc representing the isotopy class of $\alpha$, a neighborhood of $\alpha$ where $\sigma_\alpha$ is supported. Since the space of choices of $\alpha$ and its neighborhood is contractible according to Theorem~\ref{Flex}, we can neglect this ambiguity. However the space of choices of multicurves representing $\Gamma_\Sigma$ is not necessarily contractible. This point will be made clear in the next section. For the rest of this paper, $\Gamma_\Sigma$ always means a multicurve on $\Sigma$ rather than its isotopy class. \\
\end{notation}

\begin{rmk} \label{uniquetight}
If $\Sigma=S^2$ and $\Gamma_\Sigma=S^1$, then the space of choices of multicurve is simply-connected since there is a unique tight contact structure in a neighborhood of $S^2$ up to isotopy.
\end{rmk}

Observe that, up to an isotopy supported in a neighborhood of the admissible arc $\alpha$, the bypass triangle attachment does not change $\Gamma_\Sigma$.

In what follows we look at bypass triangle attachments along different admissible arcs, which leads to our alternative definition of the bypass triangle attachment.

\begin{lemma} \label{lemBTonS2}
Let $\xi_\alpha$ and $\xi_\beta$ be two (overtwisted) contact structures on $S^2\times[0,1]$, where $\alpha$ and $\beta$ are admissible arcs on $S^2\times\{0\}$, such that

\be
\item{$S^2\times\{0,1\}$ is convex with respect to both $\xi_\alpha$ and $\xi_\beta$.}
\item{$\xi_\alpha=\xi_\beta$ in a neighborhood of $S^2\times\{0\}$ and $\#\Gamma^{\xi_\alpha}_{S^2\times\{0\}}=\#\Gamma^{\xi_\beta}_{S^2\times\{0\}}=1$.}
\item{$\xi_\alpha$ is obtained by attaching a bypass triangle $\triangle_\alpha$ to $\xi_\alpha|_{S^2\times\{0\}}$, and $\xi_\beta$ is obtained by attaching a bypass triangle $\triangle_\beta$ to $\xi_\beta|_{S^2\times\{0\}}$.}
\ee

Then $\xi_\alpha$ is isotopic to $\xi_\beta$ relative to the boundary.
\end{lemma}

\begin{proof}
Up to isotopy, there are only two different admissible arcs on $(S^2\times\{0\},\xi_\alpha|_{S^2\times\{0\}})$ (or, $(S^2\times\{0\},\xi_\beta|_{S^2\times\{0\}})$). Namely, one gives the trivial bypass and the other gives the overtwisted bypass. We may assume without loss of generality that $\alpha$ is not isotopic to $\beta$, and $\sigma_\alpha$ is the trivial bypass and $\sigma_\beta$ is the overtwisted bypass. We complete the bypass triangles $\triangle_\alpha$ and $\triangle_\beta$ as depicted in Figure~\ref{BTonS2}.

\begin{figure}[h]
    \begin{overpic}[scale=.3]{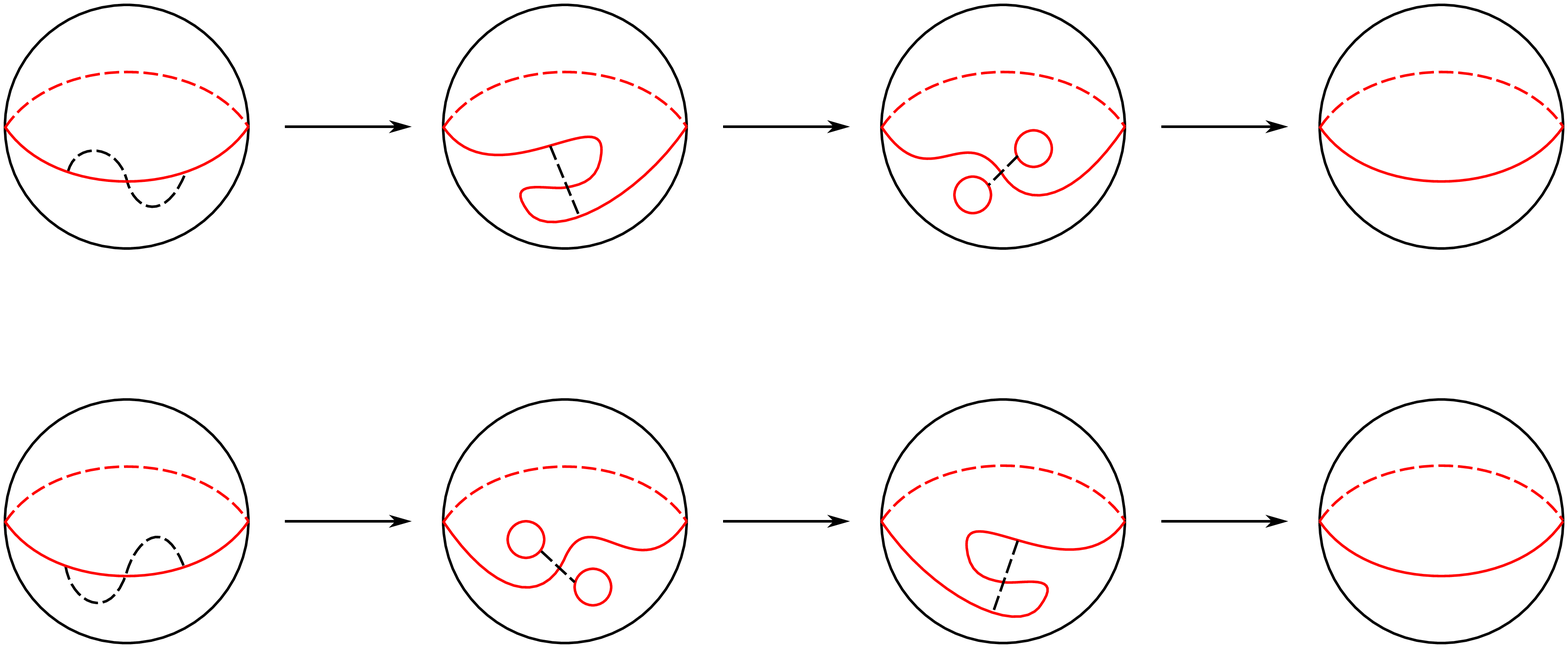}
    \put(7,32){\tiny{$\alpha$}}
    \put(32.8,29.8){\tiny{$\alpha'$}}
    \put(62.8,31.8){\tiny{$\alpha''$}}
    \put(11.5,7){\tiny{$\beta$}}
    \put(35.2,6.4){\tiny{$\beta'$}}
    \put(65,4.2){\tiny{$\beta''$}}
    \put(21,34.5){\tiny{$\sigma_\alpha$}}
    \put(48.5,34.5){\tiny{$\sigma_{\alpha'}$}}
    \put(76.5,34.5){\tiny{$\sigma_{\alpha''}$}}
    \put(21,9.5){\tiny{$\sigma_\beta$}}
    \put(48.5,9.5){\tiny{$\sigma_{\beta'}$}}
    \put(76.5,9.5){\tiny{$\sigma_{\beta''}$}}
    \end{overpic}
    \newline
    \caption{}
    \label{BTonS2}
\end{figure}

Observe that $\alpha'$ is isotopic to $\beta$, $\alpha''$ is isotopic to $\beta'$ and bypass attachments along $\alpha$ and $\beta''$ are trivial according to Lemma~\ref{Triviality}, we have the following isotopies:
\begin{align*}
\triangle_\alpha &=\sigma_\alpha \ast \sigma_{\alpha'} \ast \sigma_{\alpha''} \\
                 &\simeq \sigma_{\alpha'} \ast \sigma_{\alpha''} \\
                 &\simeq \sigma_{\beta} \ast \sigma_{\beta'} \\
                 &\simeq \sigma_\beta \ast \sigma_{\beta'} \ast \sigma_{\beta''}=\triangle_\beta.
\end{align*}
Since $S^2\times\{0,1\}$ are convex, we can make sure that the isotopies above are supported in the interior of $S^2\times[0,1]$.
\end{proof}

\begin{defn}
A {\em minimal overtwisted ball} $(B^3,\xi_{ot})$ is an overtwisted ball where $\bdry B^3$ has a tight neighborhood, and the contact structure $\xi_{ot}$ is obtained by attaching a bypass triangle to the standard tight ball $(B^3,\xi_{std})$.
\end{defn}

\begin{rmk}
By Lemma~\ref{lemBTonS2}, the minimal overtwisted ball is well-defined even if we do not specify the admissible arc along which the bypass triangle is attached.
\end{rmk}

With the above preparation, we can now redefine the bypass triangle attachment which is more convenient for our purpose. Let $(M,\xi)$ be a contact 3-manifold with convex boundary $\bdry M=\Sigma$. Identify a collar neighborhood of $\bdry M$ with $\Sigma\times[-1,0]$ such that $\bdry M=\Sigma\times\{0\}$ and the contact vector field transverse to $\bdry M$ is identified with the $[-1,0]$-direction. Let $\alpha\subset\bdry M$ be an admissible arc along which the bypass triangle is attached. Push $\alpha$ into the interior of $M$ to obtain another admissible arc, parallel to $\alpha$, contained in $\Sigma\times\{-1/2\}$, which we still denote by $\alpha$. Let $N$ be a neighborhood of $\alpha$ in $\Sigma\times\{-1/2\}$. Consider the ball with corners $N\times[-2/3,-1/3]\subset M$. By rounding the corners, we get a smoothly embedded tight ball $(B^3_1,\xi|_{B^3_1})\subset(M,\xi)$, in particular, $\bdry B^3_1$ has a tight neighborhood in $(M,\xi)$. Let $(B^3_2,\xi_{ot})$ be a minimal overtwisted ball. We construct a new contact manifold $(M,\tilde\xi)=(M\setminus B^3_1,\xi)\cup_\phi(B^3_2,\xi_{ot})$, where $\phi$ is an orientation-reversing diffeomorphism identifying the standard tight neighborhoods of $\bdry B^3_1$ and $\bdry B^3_2$. It is easy to see that $\tilde\xi$ is isotopic to the contact structure obtained by attaching a bypass triangle to $(M,\xi)$ along $\alpha$.

\begin{rmk}
The uniqueness of the tight contact structure on 3-ball, due to Eliashberg, guarantees that the bypass triangle attachment described above is well-defined.
\end{rmk}

Using the above alternative description of the bypass triangle attachment, we prove the following generalization of Lemma~\ref{lemBTonS2}.

\begin{lemma} \label{LemBT}
Let $(M,\xi)$ be a contact 3-manifold with convex boundary, and let $\alpha,\beta$ be two admissible arcs on $\bdry M$. Let $\xi_\alpha$ (resp. $\xi_\beta$) be the contact structure on $M$ obtained by attaching a bypass triangle $\triangle_\alpha$ (resp. $\triangle_\beta$) along $\alpha$ (resp. $\beta$) to $(M,\xi)$. Then $\xi_\alpha$ is isotopic to $\xi_\beta$ relative to the boundary.
\end{lemma}

\begin{proof}
Without loss of generality, we can assume that $\alpha$ and $\beta$ are disjoint. If not, we take another admissible arc $\gamma$ which is disjoint from $\alpha$ and $\beta$. We then show that $\xi_\alpha \simeq \xi_\gamma$ and $\xi_\beta \simeq \xi_\gamma$, which implies $\xi_\alpha \simeq \xi_\beta$.

As before, since $\bdry M$ is convex, we can push $\alpha$ and $\beta$ slightly into the manifold $M$, which we still denote by $\alpha$ and $\beta$. Now let $B^3_\alpha\subset M$ and $B^3_\beta\subset M$ be smoothly embedded tight balls containing $\alpha$ and $\beta$ respectively. Take a Legendrian arc $\tau$ connecting $B^3_\alpha$ and $B^3_\beta$, i.e., the endpoints of $\tau$ are contained in $\bdry B^3_\alpha$ and $\bdry B^3_\beta$, respectively, and the interior of $\tau$ is disjoint from $B^3_\alpha$ and $B^3_\beta$. Moreover, we can assume that $\tau\cap\bdry B^3_\alpha \in \Gamma_{\bdry B^3_\alpha}$ and $\tau\cap\bdry B^3_\beta \in \Gamma_{\bdry B^3_\beta}$. Let $N(\tau)$ be a closed tubular neighborhood of $\tau$. By rounding the corners of $B^3_\alpha \cup B^3_\beta \cup N(\tau)$, we get a smoothly embedded ball $B^3\subset M$ with tight convex boundary. Using our cut-and-paste definition of the bypass triangle attachment, it is easy to see that $(B^3,\xi_\alpha|_{B^3})$ and $(B^3,\xi_\beta|_{B^3})$ are isotopic, relative to the boundary, to the contact boundary sums $(B^3,\xi_{ot})\#_b(B^3,\xi_{std})$ and $(B^3,\xi_{std})\#_b(B^3,\xi_{ot})$, respectively. Hence both are isotopic to the minimal overtwisted ball. One simply extends the isotopy by identity to the rest of $M$ to conclude that $\xi_\alpha\simeq\xi_\beta$ on $M$.
\end{proof}

According to Lemma~\ref{LemBT}, the isotopy class of the contact structure obtained by attaching a bypass triangle does not depend on the choice of the attaching arcs. We shall write $\triangle$ for a bypass triangle attachment along an arbitrary admissible arc. An immediate consequence of this fact is that the bypass triangle attachment commutes with any bypass attachment. This is the content of the following corollary:

\begin{cor} \label{Commutativity}
Let $(M,\xi)$ be contact 3-manifold with convex boundary, and $\alpha$ be an admissible arc on $\bdry M$. Then $\xi\ast\sigma_\alpha \ast \triangle \simeq \xi\ast\triangle \ast \sigma_\alpha$.
\end{cor}

\begin{proof}
By Lemma~\ref{LemBT}, we can arbitrarily choose an admissible arc $\beta\subset\bdry M$ along which the bypass triangle $\triangle$ is attached. In particular, we require that $\beta$ is disjoint from $\alpha$. Hence a neighborhood of $\beta$ where $\triangle_\beta$ is supported in is also disjoint from $\alpha$. Thus we have the following isotopies:
\begin{align*}
\xi\ast\sigma_\alpha\ast\triangle &\simeq \xi\ast\sigma_\alpha\ast\triangle_\beta \\
                                      &\simeq\xi\ast \triangle_\beta\ast\sigma_\alpha \\
                                      &\simeq\xi\ast \triangle\ast\sigma_\alpha.
\end{align*}
which proves the commutativity.
\end{proof}

\begin{cor}
Let $(S^2\times[0,1],\xi)$ be a contact manifold with convex boundary, where $\xi$ is isotopic to a sequence of bypass attachments $\sigma_1\ast\sigma_2\ast\cdots\ast\sigma_n$, i.e., there exists $0=t_0<t_1<\cdots<t_n=1$ such that $S^2\times\{t_i\}$ are convex for $0\leq i\leq n$ and $S^2\times[t_{i-1},t_i]$ with the restricted contact structure is isotopic to the bypass attachment $\sigma_i$. Then $\xi\ast\triangle$ is isotopic to $\xi_k$ for $0\leq k\leq n$, where $\xi_k$ is the contact structure isotopic to a sequence of bypass attachments $\sigma_1\ast\cdots\ast\sigma_k\ast\triangle\ast\sigma_{k+1}\cdots\ast\sigma_n$.
\end{cor}

\begin{proof}
This is an iterated application of Corollary~\ref{Commutativity}.
\end{proof}

However, observe that subtracting a bypass triangle is in general not well-defined. So we need the following definition.

\begin{defn} \label{WeakIso}
Two contact structures $\xi$ and $\xi'$ on $S^2\times[0,1]$ are {\em stably isotopic}, denoted by $\xi\sim\xi'$, if they become isotopic after attaching finitely many bypass triangles to $S^2\times\{1\}$ simultaneously, i.e., $\xi \ast \triangle^n \simeq \xi' \ast \triangle^n$ for some $n\in\mathbb{N}$.
\end{defn}

\section{Overtwisted contact structures on $S^2\times[0,1]$ induced by isotopies.}

Let $\xi$ be an overtwisted contact structure on $S^2\times[0,1]$ such that $S^2\times\{0\}$ and $S^2\times\{1\}$ are convex spheres. In general, any such $\xi$ can be represented by a sequence of bypass attachments. More precisely, by Theorem~\ref{film pic}, there exists an increasing sequence $0=t_0<t_1< \cdots <t_n=1$ such that $S^2\times\{t_i\}$ is convex and $\xi|_{S^2\times[t_{i-1},t_i]}$ is isotopic to a bypass attachment $\sigma_i$ for $i=1,\cdots,n$. In this section, we consider a special class of overtwisted contact structures on $S^2\times[0,1]$ such that $S^2\times\{t\}$ is convex for $t\in[0,1]$, in other words, there is no bypass attached.

Let $\xi_0$ be an $I$-invariant contact structure on $S^2\times[0,1]$ with dividing set $\Gamma_0$ on $S^2\times\{0\}$. Let $\phi_t:S^2 \to S^2$, $t\in[0,1]$, be an isotopy such that $\phi_0=id$. We define a new contact structure $\xi_{\Gamma_0,\Phi}=\Phi_*(\xi_0)$ on $S^2\times[0,1]$, where $\Phi:S^2\times[0,1] \to S^2\times[0,1]$ is defined by $(x,t)\mapsto(\phi_t(x),t)$. Observe that $S^2\times\{t\}$ is convex with respect to $\xi_{\Gamma_0,\Phi}$ for all $t\in[0,1]$ by construction. Hence we get a smooth family of dividing sets $\Gamma_{S^2\times\{t\}}$ for $t\in[0,1]$. Conversely, a smooth family of dividing sets $\Gamma_{S^2\times\{t\}}$, $t\in[0,1]$ defines a unique contact structure on $S^2\times[0,1]$, which is isotopic to $\xi_{\Gamma_0,\Phi}$ constructed above for some isotopy $\phi_t$, $t\in[0,1]$. In practice, it is usually easier to keep track of the dividing sets rather than the isotopy.

\begin{defn}
A contact structure $\xi$ on $S^2\times[0,1]$ is {\em induced by an isotopy} if $S^2\times\{t\}$ is convex for all $t\in[0,1]$, or, equivalently, there exists an isotopy $\Phi:S^2\times[0,1] \to S^2\times[0,1]$ such that $\xi$ is isotopic to $\xi_{\Gamma_0,\Phi}$ as constructed above.
\end{defn}

It is convenient to have the following lemma.

\begin{lemma} \label{rounding}
Let $\xi$, $\xi'$ be two contact structures on $S^2\times[0,1]$ induced by isotopies and let $\Gamma_t$, $\Gamma'_t$ be dividing sets on $S^2\times\{t\}$, $0\leq t\leq1$, with respect to $\xi$, $\xi'$ respectively. If $\Gamma_0=\Gamma'_0$, $\Gamma_1=\Gamma'_1$ and there exists a path of smooth families of multicurves $\Gamma^s_t$, $0\leq s\leq 1$ satisfying the following:
\be
    \item{$\Gamma^s_t$ is a multicurve, i.e., a finite disjoint union of simple closed curves, contained in $S^2\times\{t\}$ for $0\leq s\leq1$, $0\leq t\leq1$.}
    \item{$\Gamma^0_t=\Gamma_t$, $\Gamma^1_t=\Gamma'_t$ for $0\leq t\leq1$,}
    \item{$\Gamma^s_0=\Gamma_0$, $\Gamma^s_1=\Gamma_1$ for $0\leq s\leq1$.}
\ee
then $\xi$ is isotopic to $\xi'$ relative to the boundary.
\end{lemma}

\begin{proof}
By Giroux's flexibility theorem, the path $\Gamma^s_t$, $0\leq s\leq1$ of multicurves determines a path of contact structures $\xi^s$ on $S^2\times[0,1]$ such that $\xi^0=\xi$, $\xi^1=\xi'$. Hence $\xi$ is isotopic to $\xi'$ relative to the boundary by Gray's stability theorem.
\end{proof}

We first consider a bypass attachment to the contact structures on $S^2\times[0,1]$ induced by an isotopy.

\begin{lemma} \label{DesAscBypass}
Let $\xi_{\Gamma_0,\Phi}$ be a contact structure on $S^2\times[0,1/2]$ induced by an isotopy $\phi_t:S^2 \to S^2$, $t\in[0,1/2]$, and $(S^2\times[1/2,1],\sigma_\alpha)$ be a bypass attachment along an admissible arc $\alpha\subset S^2\times\{1/2\}$. Then there exists an admissible arc $\tilde\alpha\subset S^2\times\{0\}$ such that $(S^2\times[0,1],\xi_{\Gamma_0,\Phi} \ast \sigma_\alpha)$ is isotopic, relative to the boundary, to $(S^2\times[0,1],\sigma_{\tilde\alpha} \ast \xi_{\Gamma'_0,\Phi})$, where $\Gamma'_0$ is the dividing set obtained by attaching a bypass along $\alpha$ to $\Gamma_0$.
\end{lemma}

\begin{proof}
We basically re-foliate the contact manifold $(S^2\times[0,1],\xi_{\Gamma_0,\Phi} \ast \sigma_\alpha)$. Recall that $\sigma_\alpha$ attaches a bypass $D$ on $S^2\times\{1/2\}$ so that $\bdry D=\alpha\cup\beta$ is the union of two Legendrian arcs, where $tb(\alpha)=-1$, $tb(\beta)=0$. We extend $D$ to a new bypass $\tilde D$ on $S^2\times\{0\}$ through the isotopy $\phi_t:S^2 \to S^2$, $t\in[0,1/2]$, by defining $\tilde D=D \cup \Phi(\tilde\alpha\times[0,1/2])$, where $\tilde\alpha=\phi_{1/2}^{-1}(\alpha)\subset S^2\times\{0\}$ is the new admissible arc along which $\tilde D$ is attached, and $\Phi:S^2\times[0,1/2] \to S^2\times[0,1/2]$ is defined by $(x,t)\mapsto(\phi_t(x),t)$. By attaching the new bypass $\tilde D$ on $S^2\times\{0\}$, observe that the rest of $S^2\times[0,1]$ can be foliated by convex surfaces, and the contact structure is also induced by $\Phi$. Hence $\xi_{\Gamma_0,\Phi}\ast\sigma_\alpha$ is isotopic to $\sigma_{\tilde\alpha}\ast\xi_{\Gamma'_0,\Phi}$ as desired.
\end{proof}

\begin{defn}
The admissible arc $\tilde\alpha$ constructed in Lemma~\ref{DesAscBypass} is called a {\em push-down} of $\alpha$. Conversely, we call $\alpha$ a {\em pull-up} of $\tilde\alpha$.
\end{defn}

The rest of this section is rather technical and can be skipped at the first time reading. The only result needed for our proof of Theorem~\ref{OT} is Proposition~\ref{disjointness}.

We consider a subclass of the contact structures on $S^2\times[0,1]$ induced by isotopies which we will be mainly interested in. Fix a metric on $S^2$. Without loss of generality, we assume that there exists a small disk $D^2_\epsilon(y) \subset S^2$ centered at $y$ of radius $\epsilon$ and a codimension 0 submanifold $\tilde\Gamma_{S^2\times\{0\}}$ of $\Gamma_{S^2\times\{0\}}$ such that $\tilde\Gamma_{S^2\times\{0\}} \subset D^2_\epsilon(y)$ and  $D^2_\epsilon(y) \cap \Gamma_{S^2\times\{0\}} = \tilde\Gamma_{S^2\times\{0\}}$. Let $\gamma(s)\subset S^2\times\{0\}$, $s\in[0,1]$ be an embedded oriented loop such that $\gamma(0)=\gamma(1)=y$. Let $A(\gamma)$ be an annulus neighborhood of $\gamma$ containing $D^2_\epsilon(y)$ and disjoint from other components of the dividing set as depicted in Figure~\ref{Permutation}. We define an isotopy $\phi_t:S^2 \to S^2$, $t\in[0,1]$, supported in $A(\gamma)$ which parallel transports $D^2_\epsilon(y)$ along $\gamma$ in $A(\gamma)$. More precisely, by applying the stereographic projection map, we can identify $A(\gamma)$ with an annulus in $\mathbb{R}^2$. Then the parallel transportation is given by an affine map $\phi_t:x \mapsto x+\gamma(t)-\gamma(0)$ for any $x \in D^2_\epsilon(y)$ and $t\in[0,1]$.

\begin{figure}[h]
    \begin{overpic}[scale=.3]{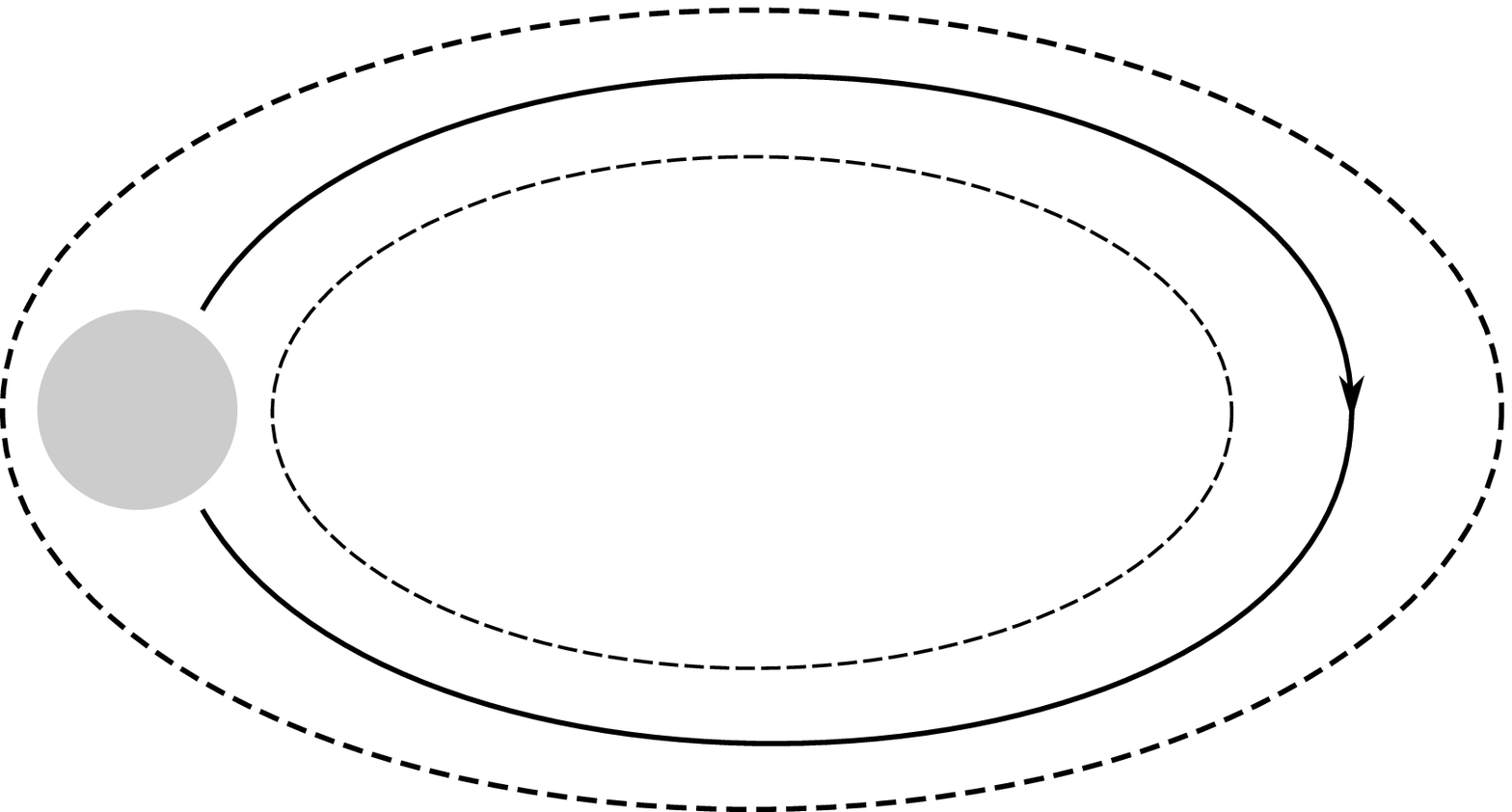}
    \put(7.5,24){\textcolor{red}{\small{$\tilde\Gamma$}}}
    \put(46,24){\textcolor{red}{\small{$\Gamma\setminus{\tilde\Gamma}$}}}
    \put(110,24){\textcolor{red}{\small{$\Gamma\setminus{\tilde\Gamma}$}}}
    \put(90,33){\small{$\gamma$}}
    \put(91,45){\small{$A(\gamma)$}}
    \end{overpic}
    \caption{}
    \label{Permutation}
\end{figure}

\begin{defn} \label{braidDef}
With the small disk $D^2_\epsilon(y) \supset \tilde\Gamma_{S^2\times\{0\}}$ such that $\tilde\Gamma_{S^2\times\{0\}} \cap \bdry D^2_\epsilon(y)=\emptyset$, the annulus $A(\gamma) \supset \gamma$ and the isotopy $\phi_t:S^2 \to S^2$ chosen as above, we say that the contact structure $\xi_{\Gamma_{S^2\times\{0\}},\Phi}$ on $S^2\times[0,1]$ is induced by a {\em pure braid of the dividing set}, where $\Phi:S^2\times[0,1] \to S^2\times[0,1]$ is induced by $\phi_t$ as before. We denote such contact structures by $\xi_{\Gamma,\Phi(\tilde\Gamma,D^2_\epsilon(y),\gamma)}$. When there is no confusion, we also abbreviate it by $\xi_{\tilde\Gamma,D^2_\epsilon,\gamma}$.
\end{defn}

\begin{rmk}
For any simply connected region $D\subset S^2\times\{0\}$ containing $\tilde\Gamma_{S^2\times\{0\}}$, one can isotop so that $D$ becomes a round disk with small radius as required in Definition~\ref{braidDef}. The isotopy class of the contact structure on $S^2\times[0,1]$ induced by a pure braid of the dividing set only depends on the choice of $D \supset \tilde\Gamma_{S^2\times\{0\}}$ and the isotopy class of $\gamma$.
\end{rmk}

\begin{rmk}
If $\xi$ is a contact structure on $S^2\times[0,1]$ induced by a pure braid of the dividing set, then $\Gamma_{S^2\times\{0\}}=\Gamma_{S^2\times\{1\}}$.
\end{rmk}

Before we give a complete classification of contact structures on $S^2\times[0,1]$ induced by pure braids of the dividing set, we make a digression into the study of its homotopy classes using a generalized version of the Pontryagin-Thom construction for manifolds with boundary. See~\cite{Hu} for more discussions on the generalized Pontryagin-Thom construction.

We can always assume that the isotopy $\phi_t(\tilde\Gamma,D^2_\epsilon(y),\gamma): S^2 \to S^2$, $t\in[0,1]$, discussed in Definition~\ref{braidDef} is supported in a disk $D^2\subset S^2$. Trivialize the tangent bundle of $D^2 \times[0,1]$ by embedding it into $\mathbb{R}^3$ so that $D^2$ is contained in the $xy$-plane. Consider the Gauss map $G:(D^2\times[0,1],\xi_{\tilde\Gamma,D^2_\epsilon,\gamma})\to S^2$. By Lemma~\ref{rounding}, we can assume without loss of generality that the dividing set is a disjoint union of round circles in $D^2\times\{t\}$ for all $0\leq t\leq1$, and $p=(1,0,0)\in S^2 \subset \mathbb{R}^3$ is a regular value. Suppose the number of connected components $\#\Gamma_{D^2\times\{0\}}=m$, then the Pontryagin submanifold $\mathcal{B}=G^{-1}(p)$ is an oriented framed monotone braid in the sense that $\mathcal{B}$ transversely intersects $D^2\times\{t\}$ in $m$ points for any $0\leq t\leq 1$, and each connected component of the dividing set contains exactly one point. It is easy to check that the pull-back framing is the blackboard framing, and consequently the self-linking number of $\mathcal{B}$ is exactly $writhe(\mathcal{B})$. It follows from the generalized Pontryagin-Thom construction that the homotopy class of a contact structure on $D^2\times[0,1]$ relative to the boundary is uniquely determined by the relative framed cobordism class of its Pontryagin submanifold $\mathcal{B}$, and hence is uniquely determined by $writhe(\mathcal{B})$ since $H_1(D^2\times[0,1],\bdry(D^2\times[0,1]);\mathbb{Z})=0$. One may think of $writhe(\mathcal{B})$ as a relative version of the Hopf invariant associated with boundary relative homotopy classes of maps $D^2\times[0,1] \simeq B^3 \to S^2$.\\

\begin{expl} \label{expl1}
{\em
If $\Gamma_{D^2\times\{0\}}$ is the disjoint union of two isolated circles, and $\tilde\Gamma_{D^2\times\{0\}}=S^1 \subset D^2_\epsilon(y)$ is the circle on the left as depicted in Figure~\ref{braidEx1}. The isotopy $\phi_t$ parallel transports $D^2_\epsilon(y)$ along the oriented loop $\gamma$. We compute the homotopy class of the contact structure $\xi_{\tilde\Gamma,D^2_\epsilon,\gamma}$.

\begin{figure}[h]
    \begin{overpic}[scale=.3]{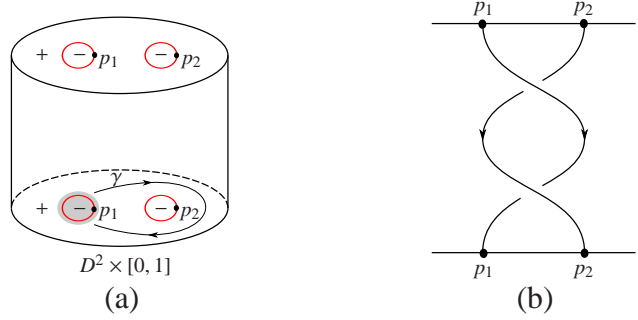}
    \put(15,-8){(a)}
    \put(81,-8){(b)}
    \put(14,6.7){\tiny{$p_1$}}
    \put(27,6.7){\tiny{$p_2$}}
    \put(14,31){\tiny{$p_1$}}
    \put(27.2,31){\tiny{$p_2$}}
    \put(74,-2.5){\tiny{$p_1$}}
    \put(90,-2.5){\tiny{$p_2$}}
    \put(74,39.5){\tiny{$p_1$}}
    \put(90,39.5){\tiny{$p_2$}}
    \put(4,7){\tiny{$+$}}
    \put(4,31.7){\tiny{$+$}}
    \put(10,7){\tiny{$-$}}
    \put(23,7){\tiny{$-$}}
    \put(10,31.7){\tiny{$-$}}
    \put(23,31.7){\tiny{$-$}}
    \put(11,-2.3){\tiny{$D^2\times[0,1]$}}
    \put(16,12.2){\tiny{$\gamma$}}
    \end{overpic}
    \newline
    \caption{(a) The contact structure on $S^2\times[0,1]$ induced by a full twist of the dividing circles, where $\{p_1,p_2\}$ are pre-images of the regular value $p=(1,0,0)\in S^2$. (b) The oriented braid with the blackboard framing $\mathcal{B}$ as the Pontryagin submanifold.}
    \label{braidEx1}
\end{figure}

According to the Pontryagin-Thom construction, since $writhe(\mathcal{B})=-2$, the homotopy class of $\xi_{\tilde\Gamma,D^2_\epsilon,\gamma}$ is in general different from the $I$-invariant contact structure, and the difference is measured by decreasing the Hopf invariant by 2.\footnote{However, if the divisibility of the Euler class is 2, then $\phi_t$ gives a contact structure which is homotopic to the $I$-invariant contact structure. We will discuss the divisibility of the Euler class in detail in Section 8.}
}\\
\end{expl}

\begin{expl} \label{expl2}
{\em
If $\Gamma_{D^2\times\{0\}}$ is the disjoint union of three circles, and $\tilde\Gamma_{D^2\times\{0\}}=S^1 \subset D^2_\epsilon(y)$ is the circle on the left as depicted in Figure~\ref{braidEx2}. The isotopy $\phi_t$ parallel transports $D^2_\epsilon(y)$ along the oriented loop $\gamma$. We compute the homotopy class of the contact structure $\xi_{\tilde\Gamma,D^2_\epsilon,\gamma}$.

\begin{figure}[h]
    \begin{overpic}[scale=.35]{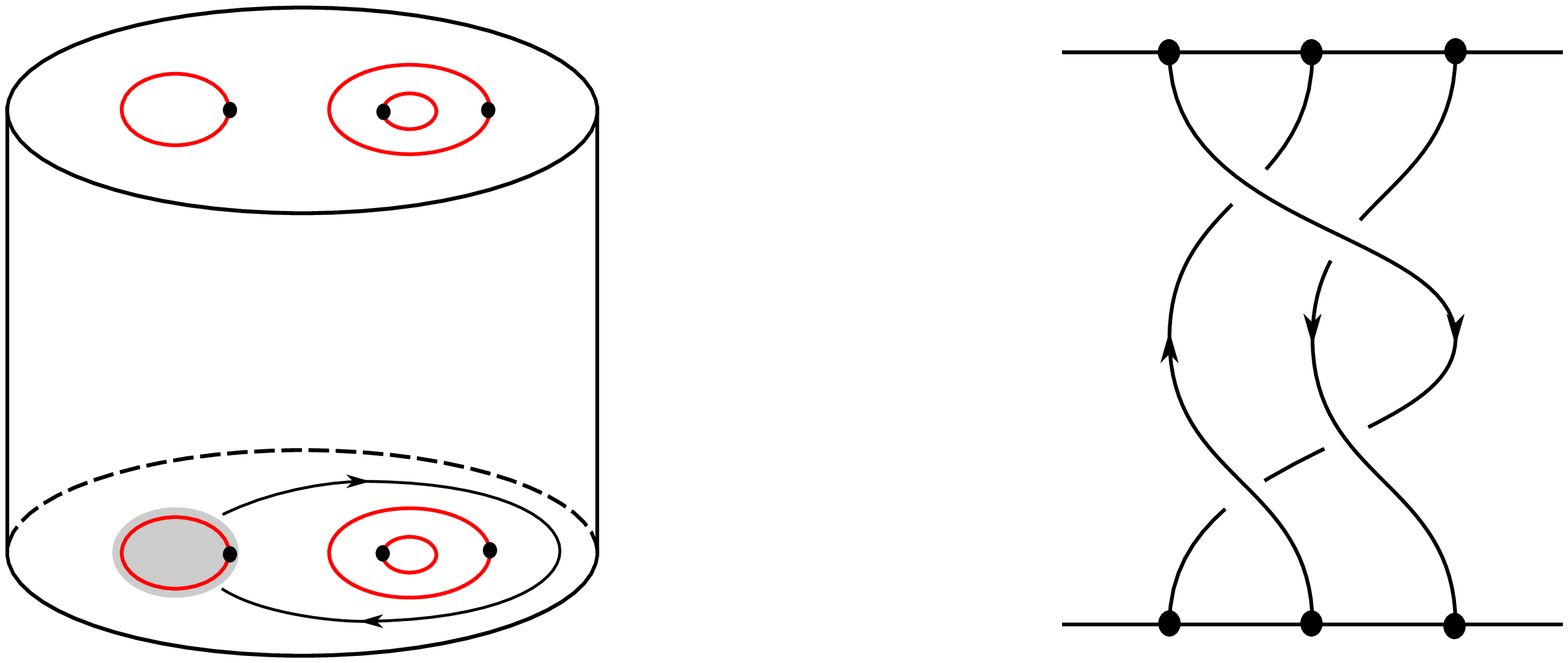}
    \put(14.5,5.2){\tiny{$p_1$}}
    \put(20.5,5.2){\tiny{$p_2$}}
    \put(31.3,5.2){\tiny{$p_3$}}
    \put(14.5,33.2){\tiny{$p_1$}}
    \put(20.8,33.2){\tiny{$p_2$}}
    \put(31.3,33.2){\tiny{$p_3$}}
    \put(15,-8){(a)}
    \put(81,-8){(b)}
    \put(12,-3.3){\tiny{$D^2\times[0,1]$}}
    \put(73,-2){\tiny{$p_1$}}
    \put(82,-2){\tiny{$p_2$}}
    \put(92,-2){\tiny{$p_3$}}
    \put(73,42){\tiny{$p_1$}}
    \put(82,42){\tiny{$p_2$}}
    \put(92,42){\tiny{$p_3$}}
    \put(3.5,5.5){\tiny{$+$}}
    \put(10,5.5){\tiny{$-$}}
    \put(24.8,5.66){\tiny{$+$}}
    \put(28,5.5){\tiny{$-$}}
    \put(3.5,34.3){\tiny{$+$}}
    \put(10,34.3){\tiny{$-$}}
    \put(24.8,34.3){\tiny{$+$}}
    \put(28,34.3){\tiny{$-$}}
    \put(17,11.5){\tiny{$\gamma$}}
    \end{overpic}
    \newline
    \caption{(a) A braiding by a full twist of the left-hand side dividing circle along $\gamma$, where $\{p_1,p_2,p_3\}=G^{-1}(p)$ is the pre-image of the regular value $p=(1,0,0)\in S^2$. (b) The oriented framed braid $\mathcal{B}$ as the Pontryagin submanifold.}
    \label{braidEx2}
\end{figure}

In this case, one computes that $writhe(\mathcal{B})=0$, hence $\xi_{\tilde\Gamma,D^2_\epsilon,\gamma}$ is homotopic to the $I$-invariant contact structure.
}
\end{expl}

Now we are ready to classify the contact structures induced by pure braids of the dividing set up to stable isotopy in the sense of Definition~\ref{braidDef}. One goal is to establish an isotopy equivalence relation between a pure braid of the dividing set and the bypass triangle attachment. To start with, we consider the contact structures induced by two special pure braids of the dividing set as depicted in Figure~\ref{Permute}. In Figure~\ref{Permute}(a), the dividing set $\tilde\Gamma \subset D^2_\epsilon(y)$ is a single circle, and the dividing set contained in the disk bounded by $\gamma$ and disjoint from $\tilde\Gamma$ is also a single circle. In Figure~\ref{Permute}(b), the dividing set $\tilde\Gamma \subset D^2_\epsilon(y)$ consists of $m$ isolated circles nested in another circle, and the dividing set contained in the disk bounded by $\gamma$ and disjoint from $\tilde\Gamma$ consists of $n$ isolated circles nested in another circle. We also assume that either $m$ or $n$ is not zero. For technical reasons, it is convenient to have the following definitions.

\begin{defn}
Given two disjoint embedded circles $\gamma,\gamma' \subset D^2$, we say $\gamma<\gamma'$ if and only if $\gamma$ is contained in the disk bounded by $\gamma'$.
\end{defn}

\begin{defn}
Let $\Gamma \subset D^2$ be a finite disjoint union of embedded circles. The {\em depth} of $\Gamma$ is the maximum length of chains $\gamma_1<\gamma_2<\cdots<\gamma_r$, where $\gamma_i\subset\Gamma$ is a single circle for any $i\in\{1,2,\cdots,r\}$.
\end{defn}

Observe that the depth of the dividing set in Figure~\ref{Permute}(a) is 1, and the depth of the dividing set in Figure~\ref{Permute}(b) is 2. It turns out that to study the contact structure induced by an arbitrary pure braid of the dividing set, it suffices to consider a finite composition of these two special cases.

\begin{figure}[h]
    \begin{overpic}[scale=.25]{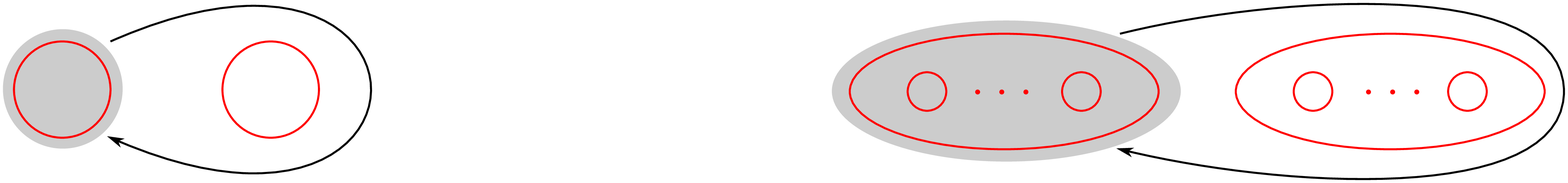}
    \put(10,-4.5){(a)}
    \put(77,-4.5){(b)}
    \put(22.5,10){\small{$\gamma$}}
    \put(100,10){\small{$\gamma$}}
    \put(-2.2,2){\small{$\Gamma'$}}
    \put(51.5,2){\small{$\Gamma'$}}
    \put(60,4.5){$\underbrace{}$}
    \put(84.7,4.5){$\underbrace{}$}
    \put(63.2,.8){\tiny{$m$}}
    \put(88.5,.8){\tiny{$n$}}
    \end{overpic}
    \newline
    \caption{}
    \label{Permute}
\end{figure}

\begin{lemma} \label{SimplePermA}
If $(S^2\times[0,1],\xi_{\tilde\Gamma,D^2_\epsilon,\gamma})$ is a contact manifold with contact structure induced by a pure braid of the dividing set where $\tilde\Gamma \subset D^2_\epsilon$ and $\gamma$ are chosen as in Figure~\ref{Permute}(a), then $(S^2\times[0,1],\xi_{\tilde\Gamma,D^2_\epsilon,\gamma})$ is isotopic relative to the boundary to $(S^2\times[0,1],\triangle^2)$, where $\triangle^2$ denotes the contact structure obtained by attaching two bypass triangles on $(S^2\times\{0\},\xi_{\tilde\Gamma,D^2_\epsilon,\gamma}|_{S^2\times\{0\}})$.
\end{lemma}

\begin{proof}
Let $\alpha$ be an admissible arc as depicted in Figure~\ref{braidLem1}(b). Suppose that both bypass triangles are attached along $\alpha$.

\begin{figure}[h]
    \begin{overpic}[scale=.2]{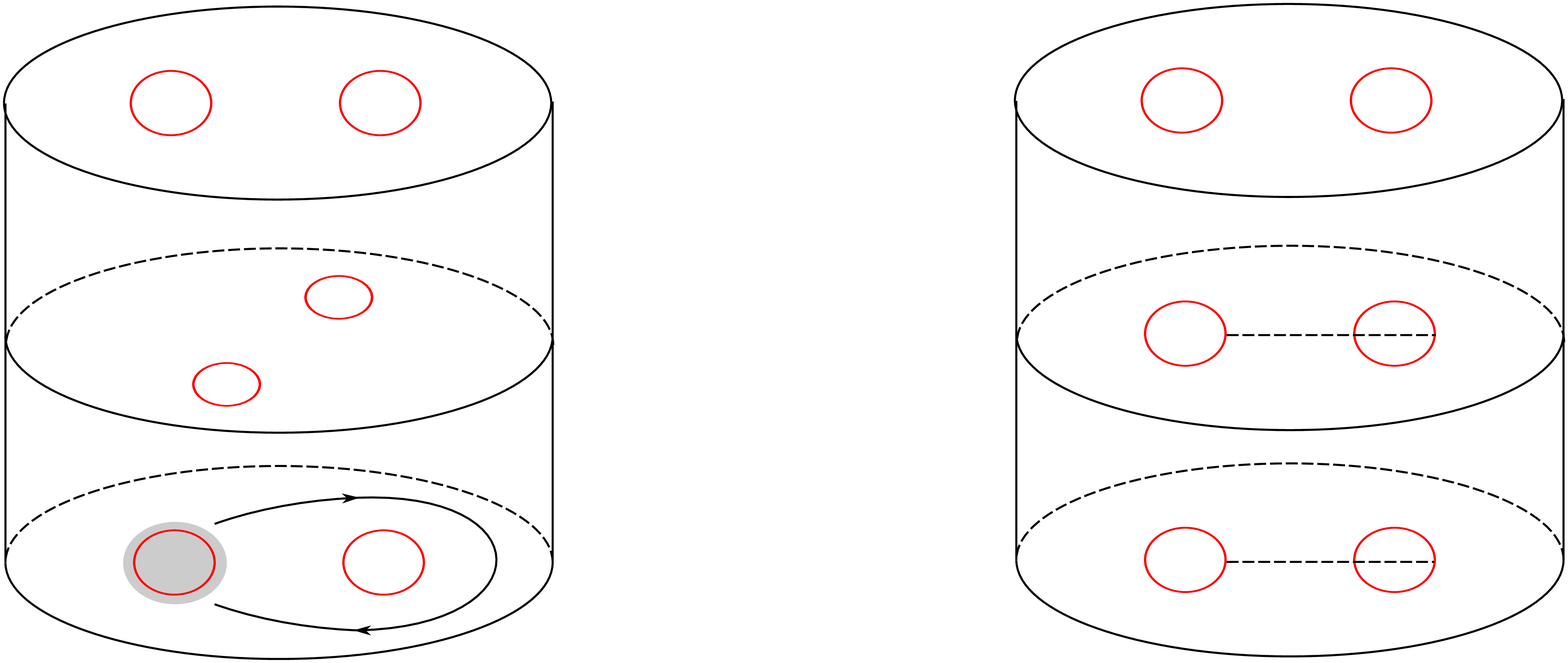}
    \put(14,-7){(a)}
    \put(80,-7){(b)}
    \put(16,.5){\tiny{$\gamma$}}
    \put(82,7){\tiny{$\alpha$}}
    \put(82,21.5){\tiny{$\alpha$}}
    \put(38,20){$\xi_{\tilde\Gamma,D^2_\epsilon,\gamma}$}
    \put(102,11){$\triangle_\alpha$}
    \put(102,24){$\triangle_\alpha$}
    \end{overpic}
    \newline
    \caption{(a) The contact structure is induced by parallel transporting $\tilde\Gamma \subset D^2_\epsilon$ along $\gamma$. (b) Attaching two bypass triangles along the admissible arc $\alpha$.}
    \label{braidLem1}
\end{figure}

Observe that $\triangle_\alpha=\sigma_{\alpha} \ast \sigma_{\alpha'} \ast \sigma_{\alpha''}$, where $\sigma_{\alpha}$, $\sigma_{\alpha'}$ and $\sigma_{\alpha''}$ are all trivial bypass attachments. Hence the contact manifold $(S^2\times[0,1],\triangle_{\alpha}^2)$ can be foliated by convex surfaces by Lemma~\ref{Triviality}. In other words, it is induced by an isotopy. By Theorem 0.5\footnote{The 3-dimensional obstruction class $o_3$ used in Theorem 0.5 in~\cite{Hu} is by definition the relative version of the Hopf invariant which we have discussed above.} in~\cite{Hu}, we know that attaching two bypass triangles $\triangle_{\alpha}^2$ decreases the Hopf invariant by 2. In Example~\ref{expl1}, we checked by Pontryagin-Thom construction that $\xi_{\tilde\Gamma,D^2_\epsilon,\gamma}$ also decreases the Hopf invariant by 2. Observe that the isotopy class relative to the boundary of a 2-strand oriented monotone braid with blackboard framing is uniquely determined by its self-linking number, which is equal to the Hopf invariant. Hence $\triangle^2_{\alpha}$ is isotopic $\Phi_{\tilde\Gamma,D^2_\epsilon,\gamma}$ in the region where both operations are supported. By extending the isotopy by identity to the rest of $S^2$, we conclude that $(S^2\times[0,1],\xi_{\tilde\Gamma,D^2_\epsilon,\gamma})$ is isotopic relative to the boundary to $(S^2\times[0,1],\triangle^2)$.
\end{proof}

\begin{lemma} \label{SimplePermBCD}
If $(S^2\times[0,1],\xi_{\tilde\Gamma,D^2_\epsilon,\gamma})$ is a contact manifold with contact structure induced by a pure braid of the dividing set where $\tilde\Gamma \subset D^2_\epsilon$ and $\gamma$ are chosen as in Figure~\ref{Permute}(b), then $(S^2\times[0,1],\xi_{\tilde\Gamma,D^2_\epsilon,\gamma})$ is stably isotopic to $(S^2\times[0,1],\triangle^{2(m-1)(n-1)})$.
\end{lemma}

\begin{proof}
Let $\alpha\subset S^2\times\{1\}$ be an admissible arc as depicted in the left-hand side of Figure~\ref{braidLem2}(a). By Lemma~\ref{DesAscBypass}, if $\tilde\alpha$ is the push-down of $\alpha$, then $\xi_{\Gamma,\Phi(\tilde\Gamma,D^2_\epsilon,\gamma)}\ast\sigma_\alpha \simeq \sigma_{\tilde\alpha}\ast\xi_{\Gamma',\Phi}$, where $\Gamma'$ is obtained from $\Gamma$ by attaching a bypass along $\alpha$. We remark here that $\xi_{\Gamma,\Phi(\tilde\Gamma,D^2_\epsilon,\gamma)}$ and $\xi_{\Gamma',\Phi}$ are contact structures induced by the same isotopy, but are push-forwards of different contact structures on $S^2\times[0,1]$. Choose $\tilde\Gamma'\subset D^{2'}_\epsilon$ to be the $m$ isolated circles on the left and $\gamma'$ be an oriented loop as depicted in the right-hand side of Figure~\ref{braidLem2}(a). Let $\xi_{\tilde\Gamma',D^{2'}_\epsilon,\gamma'}$ be the contact structure induced by an isotopy which parallel transports $\tilde\Gamma'\subset D^{2'}_\epsilon$ along $\gamma'$. Then Lemma~\ref{rounding} implies that $\xi_{\Gamma',\Phi}$ is isotopic, relative to the boundary, to $\xi_{\tilde\Phi}\ast\xi_{\tilde\Gamma',D^{2'}_\epsilon,\gamma'}$, where $\tilde\Phi$ is induced by an isotopy that rounds the outmost dividing circle. An iterated application of Lemma~\ref{SimplePermA} implies that $\xi_{\tilde\Gamma',D^{2'}_\epsilon,\gamma'} \simeq \triangle^{2m(n-1)}$.

We next isotop the contact structure $\sigma_{\tilde\alpha}\ast\xi_{\tilde\Phi}$. Consider the $n$ isolated circles nested in a larger circle. Let $\tilde\Gamma''\subset D^{2''}_\epsilon$ be the leftmost circle among the $n$ circles and $\gamma''$ be an oriented loop as depicted in the right-hand side of Figure~\ref{braidLem2}(b). We pull up $\tilde\alpha$ through an isotopy which parallel transports $\tilde\Gamma''\subset D^{2''}_\epsilon$ along $\gamma''$, and observe that the pull-up of $\tilde\alpha$ is isotopic to $\alpha$. By using Lemma~\ref{DesAscBypass} one more time, we get the isotopy of contact structures $\sigma_{\tilde\alpha}\ast\xi_{\tilde\Phi} \simeq \xi_{\tilde\Gamma'',D^{2''}_\epsilon,\gamma''}\ast\sigma_\alpha$. It is left to determine the isotopy class of the contact structure $\xi_{\tilde\Gamma'',D^{2''}_\epsilon,\gamma''}$. Since $\gamma''$ is oriented counterclockwise, by applying Lemma~\ref{SimplePermA} $(n-1)$ times, we get a stable isotopy $\xi_{\tilde\Gamma'',D^{2''}_\epsilon,\gamma''} \sim \triangle^{2(1-n)}$, i.e., $\xi_{\tilde\Gamma'',D^{2''}_\epsilon,\gamma''}\ast\triangle^{2(n-1)}$ is isotopic to the $I$-invariant contact structure.

\begin{figure}[h]
    \begin{overpic}[scale=.2]{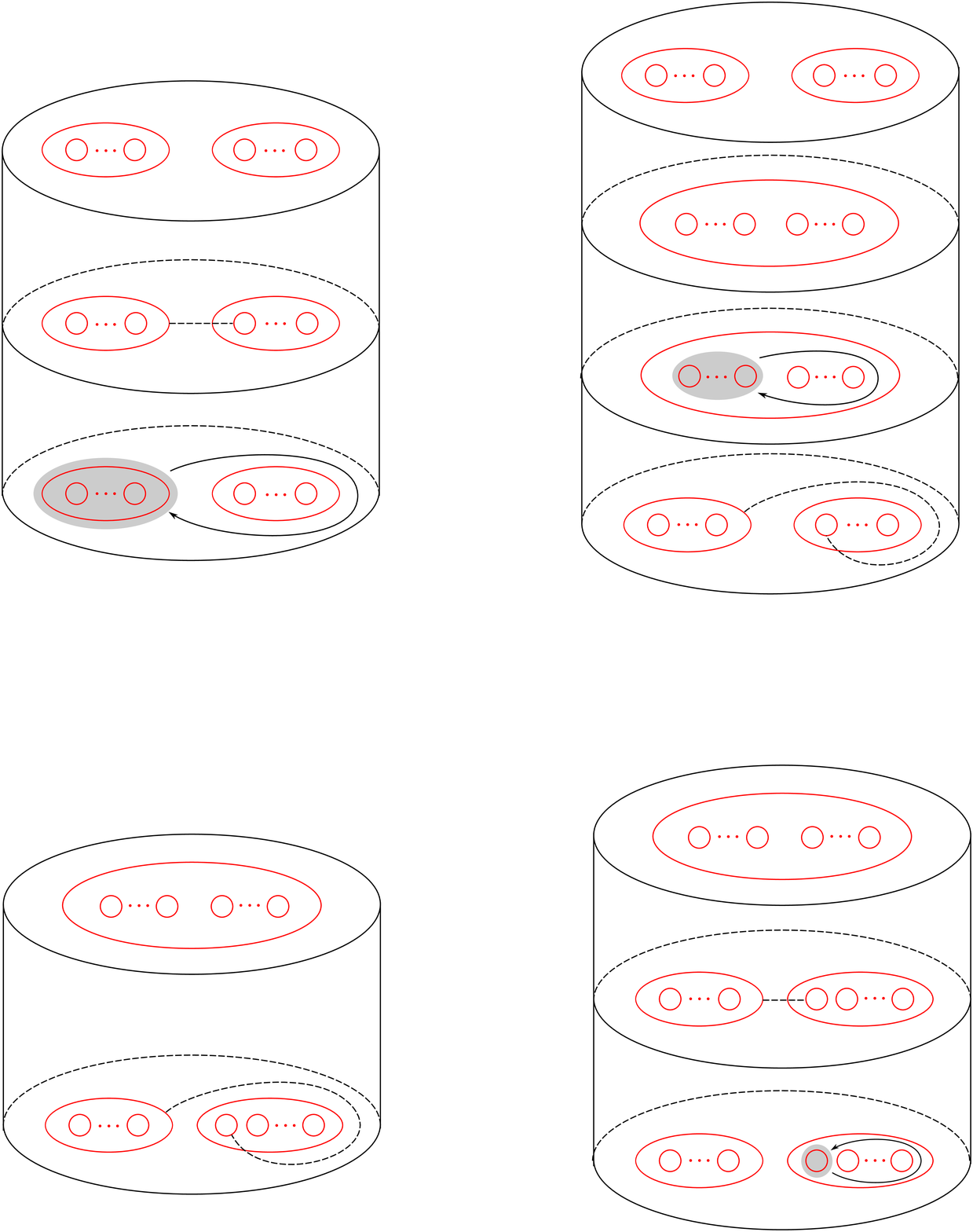}
    \put(38,47){(a)}
    \put(38,-4){(b)}
    \put(38,18){$\simeq$}
    \put(38,73){$\simeq$}
    \put(32,13){$\sigma_{\tilde\alpha}$}
    \put(80,10){$\xi_{\tilde\Gamma'',D^{2''}_\epsilon,\delta''}$}
    \put(80,24){$\sigma_\alpha$}
    \put(19,4){\tiny{$\tilde\alpha$}}
    \put(69,2){\tiny{$\gamma''$}}
    \put(62.5,19){\tiny{$\alpha$}}
    \put(32,64  ){$\xi_{\tilde\Gamma,D^2_\epsilon,\gamma}$}
    \put(32,78){$\triangle_\alpha$}
    \put(15,56.2){\tiny{$\gamma$}}
    \put(15,74.5){\tiny{$\alpha$}}
    \put(80,62){$\sigma_{\tilde\alpha}$}
    \put(80,74){$\xi_{\tilde\Gamma',D^{2'}_\epsilon,\gamma'}$}
    \put(80,86){$\sigma_{\alpha'} \ast \sigma_{\alpha''}$}
    \put(68,52.8){\tiny{$\tilde\alpha$}}
    \put(62.5,65.8){\tiny{$\gamma'$}}
    \put(8,55.3){\tiny{$\tilde\Gamma$}}
    \put(57.6,66.5){\tiny{$\tilde\Gamma'$}}
    \put(65.7,2.5){\tiny{$\tilde\Gamma''$}}
    \end{overpic}
    \newline
    \caption{(a) Pushing down the bypass attachment $\sigma_\alpha$. (b) Pulling up the bypass attachment $\sigma_{\tilde\alpha}$.}
    \label{braidLem2}
\end{figure}

To summarize what we have done so far, we have the following (stable) isotopies of contact structures:
\begin{align*}
\xi_{\tilde\Gamma,D^2_\epsilon,\gamma}\ast\triangle_\alpha
&=\xi_{\Gamma,\Phi(\tilde\Gamma,D^2_\epsilon,\gamma)}\ast\sigma_\alpha\ast\sigma_{\alpha'}\ast\sigma_{\alpha''} \\
&\simeq\sigma_{\tilde\alpha}\ast\xi_{\Gamma',\Phi}\ast\sigma_{\alpha'}\ast\sigma_{\alpha''} \\
&\simeq\sigma_{\tilde\alpha}\ast\xi_{\tilde\Phi}\ast\xi_{\tilde\Gamma',D^{2'}_\epsilon,\gamma'}\ast\sigma_{\alpha'}\ast\sigma_{\alpha''} \\
&\simeq\sigma_{\tilde\alpha}\ast\xi_{\tilde\Phi}\ast\triangle^{2m(n-1)}\ast\sigma_{\alpha'}\ast\sigma_{\alpha''} \\
&\simeq\xi_{\tilde\Gamma'',D^{2''}_\epsilon,\gamma''}\ast\sigma_\alpha\ast\triangle^{2m(n-1)}\ast\sigma_{\alpha'}\ast\sigma_{\alpha''} \\
&\sim\triangle^{2(1-n)}\ast\sigma_\alpha\ast\triangle^{2m(n-1)}\ast\sigma_{\alpha'}\ast\sigma_{\alpha''} \\
&\simeq\triangle^{2(m-1)(n-1)}\ast\sigma_\alpha\ast\sigma_{\alpha'}\ast\sigma_{\alpha''} \\
&=\triangle^{2(m-1)(n-1)}\ast\triangle_\alpha.
\end{align*}

Note that the third equation from the bottom is only a stable isotopy so that the (possibly) negative power of the bypass triangle attachment makes sense. See Definition~\ref{WeakIso}. We will use the same trick in the proof of the following Proposition~\ref{PermEquToBTs} without further mentioning. Hence by definition, $\xi_{\tilde\Gamma,D^2_\epsilon,\gamma}$ is stably isotopic to $\triangle^{2(m-1)(n-1)}$ as desired.
\end{proof}

We now completely classify contact structures on $S^2\times[0,1]$ induced by pure braids of the dividing set.

\begin{prop} \label{PermEquToBTs}
If $(S^2\times[0,1],\xi_{\tilde\Gamma,D^2_\epsilon,\gamma})$ is a contact manifold with contact structure induced by a pure braid of the dividing set, then $\xi_{\tilde\Gamma,D^2_\epsilon,\gamma}$ is stably isotopic to $(S^2\times[0,1],\triangle^{l})$ for some $l\in\mathbb{N}$.
\end{prop}

\begin{proof}
Recall that $\tilde\Gamma\subset D^2_\epsilon$ is a codimension 0 submanifold of $\Gamma_{S^2\times\{0\}}$, and $\gamma$ is an oriented loop in the complement of $\Gamma_{S^2\times\{0\}}$ as in Definition~\ref{braidDef}. Let $\tilde\Gamma'$ be the union of components of $\Gamma_{S^2\times\{0\}}$ contained in a disk bounded by $\gamma$ and outside of $A(\gamma)$. We may choose the disk so that $-\gamma$ is the oriented boundary. Since the contact structure $\xi_{\tilde\Gamma,D^2_\epsilon,\gamma}$ is induced by a pure braid of the dividing set, we have $\Gamma_{S^2\times\{0\}}=\Gamma_{S^2\times\{1\}}$. Hence we also view $\tilde\Gamma$ and $\tilde\Gamma'$ as dividing sets on $S^2\times\{1\}$. Choose pairwise disjoint admissible arcs $\alpha_1,\alpha_2,\cdots,\alpha_r,\alpha_{r+1},\cdots,\alpha_k$ on $S^2\times\{1\}$ such that the following conditions hold:

\be
\item{$\alpha_1,\alpha_2,\cdots,\alpha_{r-1}$ are admissible arcs contained in $D^2_\epsilon$ such that by attaching bypasses along these arcs, the depth of $\tilde\Gamma$ becomes at most 2.}
\item{$\alpha_r,\alpha_{r+1},\cdots,\alpha_k$ are admissible arcs contained in the disk bounded by $\gamma$ and outside of $A(\gamma)$ such that by attaching bypasses along these arcs, the depth of $\tilde\Gamma'$ becomes at most 2.}
\ee

Observe that we choose $\alpha_1,\alpha_2,\cdots,\alpha_k$ such that the isotopy class of each $\alpha_i$ is invariant under the time-1 map $\phi_1$ which is supported in $A(\gamma) \setminus D^2_\epsilon$. Hence, by abuse of notation, we do not distinguish $\alpha_i$ and its push-down through $\phi_t(\tilde\Gamma,D^2_\epsilon,\gamma)$. By Lemma~\ref{DesAscBypass},  we have the isotopy of contact structures $\xi_{\tilde\Gamma,D^2_\epsilon,\gamma}\ast\sigma_{\alpha_1}\ast\cdots\ast\sigma_{\alpha_k} \simeq \sigma_{\alpha_1}\ast\cdots\ast\sigma_{\alpha_k}\ast\xi_\Phi$, where $\xi_\Phi$ is the contact structure induced by a finite composition of special pure braids of the dividing set considered in Lemma~\ref{SimplePermA} and Lemma~\ref{SimplePermBCD}, Therefore $\xi_\Phi$ is stable isotopic to a power of the bypass triangle attachment, say $\triangle^l$ for some $l\in\mathbb{N}$. To summarize, we have the following (stable) isotopies of contact structures, relative to the boundary.

\begin{align*}
\xi_{\tilde\Gamma,D^2_\epsilon,\gamma}\ast\triangle^k
&\simeq\xi_{\tilde\Gamma,D^2_\epsilon,\gamma}\ast\triangle_{\alpha_1}\ast\cdots\ast\triangle_{\alpha_k} \\ &=\xi_{\tilde\Gamma,D^2_\epsilon,\gamma}\ast(\sigma_{\alpha_1}\ast\sigma_{\alpha'_1}\ast\sigma_{\alpha''_1})\ast\cdots\ast(\sigma_{\alpha_k}\ast\sigma_{\alpha'_k}\ast\sigma_{\alpha''_k}) \\
&\simeq(\xi_{\tilde\Gamma,D^2_\epsilon,\gamma}\ast\sigma_{\alpha_1}\ast\cdots\ast\sigma_{\alpha_k})\ast(\sigma_{\alpha'_1}\ast\sigma_{\alpha''_1})\ast\cdots\ast(\sigma_{\alpha'_k}\ast\sigma_{\alpha''_k}) \\
&\simeq(\sigma_{\alpha_1}\ast\cdots\ast\sigma_{\alpha_k}\ast\xi_{\Phi})\ast(\sigma_{\alpha'_1}\ast\sigma_{\alpha''_1})\ast\cdots\ast(\sigma_{\alpha'_k}\ast\sigma_{\alpha''_k}) \\
&\sim(\sigma_{\alpha_1}\ast\cdots\ast\sigma_{\alpha_k}\ast\triangle^l)\ast(\sigma_{\alpha'_1}\ast\sigma_{\alpha''_1})\ast\cdots\ast(\sigma_{\alpha'_k}\ast\sigma_{\alpha''_k}) \\
&\simeq\triangle^l\ast(\sigma_{\alpha_1}\ast\sigma_{\alpha'_1}\ast\sigma_{\alpha''_1})\ast\cdots\ast(\sigma_{\alpha_k}\ast\sigma_{\alpha'_k}\ast\sigma_{\alpha''_k}) \\
&=\triangle^l\ast\triangle^k.
\qedhere
\end{align*}

Hence $\xi_{\tilde\Gamma,D^2_\epsilon,\gamma}$ is stably isotopic to $\triangle^l$ by definition.

\end{proof}

To conclude this section, we prove the following technical result which asserts that under certain assumptions and up to possible bypass triangle attachments, one can separate two bypasses.

\begin{prop} \label{disjointness}
Let $(S^2,\Gamma)$ be a convex sphere with dividing set $\Gamma$ and $\alpha\subset(S^2,\Gamma)$ be an admissible arc such that the bypass attachment $\sigma_\alpha$ increases $\#\Gamma$ by 2. Suppose that $(S^2,\Gamma')$ is the new convex sphere obtained by attaching $\sigma_\alpha$ to $(S^2,\Gamma)$ and suppose $\beta\subset(S^2,\Gamma')$ is another admissible arc such that the bypass attachment $\sigma_\beta$ decreases $\#\Gamma'$ by 2. Then there exists an admissible arc $\tilde\beta\subset(S^2,\Gamma)$ disjoint from $\alpha$, a map $\Phi:S^2\times[0,1] \to S^2\times[0,1]$ induced by an isotopy, and an integer $l\in\mathbb{N}$ such that $\sigma_\alpha\ast\sigma_\beta \sim \sigma_\alpha\ast\sigma_{\tilde\beta}\ast\triangle^l\ast\xi_\Phi$ relative to the boundary.
\end{prop}

\begin{proof}
Let $\delta$ be the arc of anti-bypass attachment to $\sigma_\alpha$ contained in $(S^2,\Gamma')$ as discussed in Remark~\ref{anibypassarc}. Then $\delta$ intersects $\Gamma'$ in three points $\{p_1,p_2,p_3\}$ as depicted in Figure~\ref{antibypass}(b). Let $\delta_1$ and $\delta_2$ be subarcs of $\delta$ from $p_1$ to $p_2$ and from $p_2$ to $p_3$ respectively. Observe that, in order to find an admissible arc $\tilde\beta \subset (S^2,\Gamma)$ which is disjoint from $\alpha$ and satisfy all the conditions in the lemma, it suffices to find an admissible arc on $(S^2,\Gamma')$, which we still denote by $\tilde\beta$, and which is disjoint from $\delta$ and also satisfies the conditions in the lemma. In fact, by symmetry, we only need $\tilde\beta$ to be disjoint from $\delta_1$. Without loss of generality, we can assume that $\beta$ intersects $\delta$ transversely and the intersection points are different from $p_1$, $p_2$ and $p_3$.\\

\begin{figure}[h]
    \begin{overpic}[scale=.35]{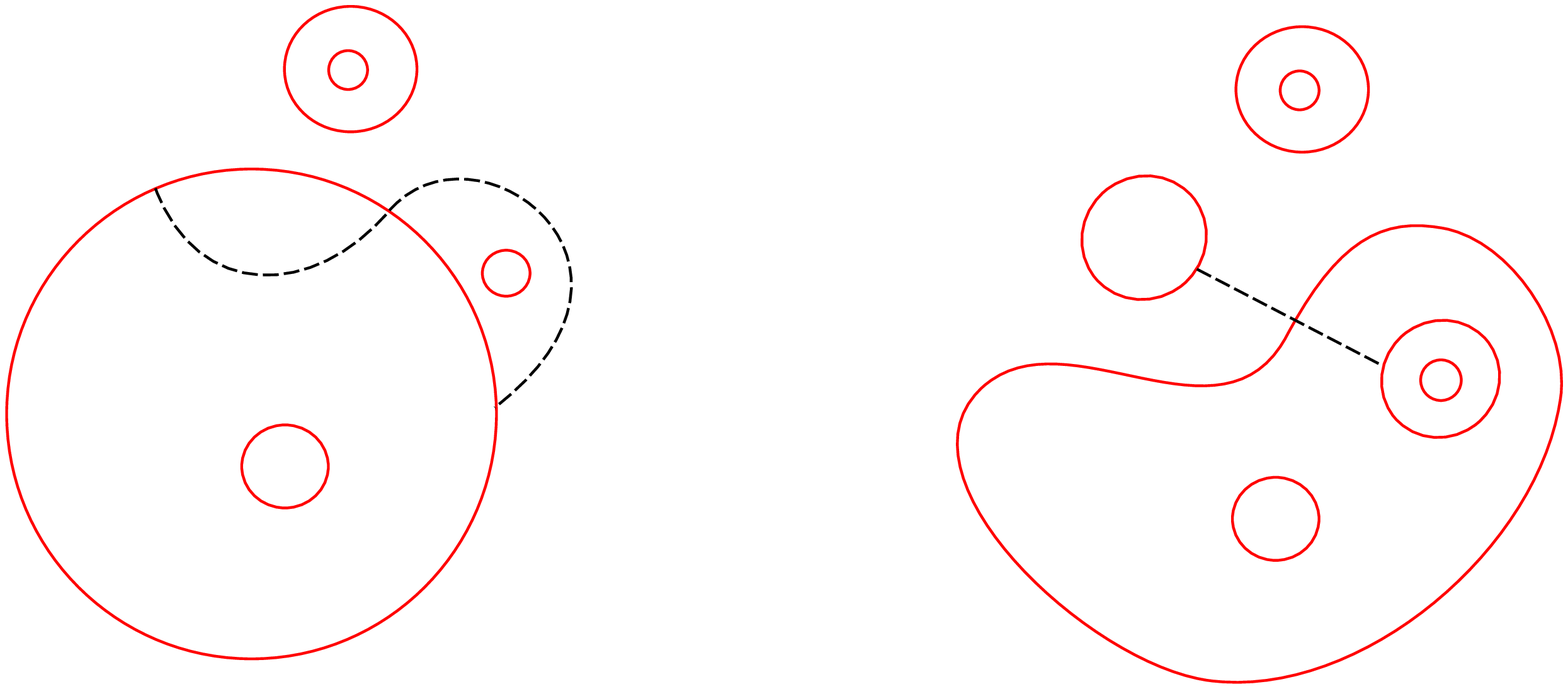}
    \put(34.3,31){\tiny{$\alpha$}}
    \put(85,23){\tiny{$\delta$}}
    \put(73.7,27){\tiny{$p_1$}}
    \put(80,22){\tiny{$p_2$}}
    \put(88,17.9){\tiny{$p_3$}}
    \put(11,-5){(a)}
    \put(83,-5){(b)}
    \end{overpic}
    \newline
    \caption{(a) The convex sphere $(S^2,\Gamma)$ with an admissible arc $\alpha$. (b) The convex sphere $(S^2,\Gamma')$ obtained by attaching a bypass along $\alpha$, where $\delta$ is the arc of the anti-bypass attachment.}
    \label{antibypass}
\end{figure}

\noindent \emph {Claim: Up to isotopy and possibly a finite number of bypass triangle attachments, one can arrange so that $\beta$ and $\delta_1$ do not cobound a bigon $B$ on $S^2$ as depicted in Figure~\ref{cancelbigon}(a)}.\\

\begin{figure}[h]
    \begin{overpic}[scale=.23]{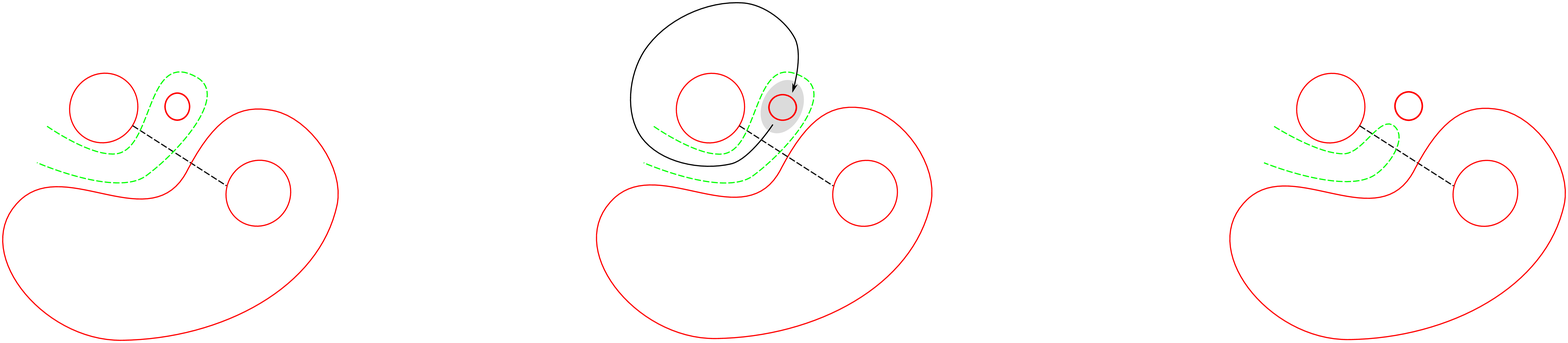}
    \put(8.5,11.5){\tiny{$\delta_1$}}
    \put(13.5,16.5){\tiny{$\beta$}}
    \put(38.5,16){\tiny{$\gamma$}}
    \put(79.3,13){\tiny{$\tilde\beta$}}
    \put(9,-5){(a)}
    \put(48,-5){(b)}
    \put(88,-5){(c)}
    \end{overpic}
    \newline
    \caption{(a) The admissible arc $\beta$ together with $\delta_1$ bound a minimal bigon, which contains other components of the dividing set in the interior. (b) Choose a disk $D^2_\epsilon$ containing all the dividing sets $\tilde\Gamma$ in the bigon and an oriented loop $\gamma$ so that it intersects $\beta$ in exactly one point. (c) The pull-up of $\beta$ through the contact structure $\xi_{\tilde\Gamma,D^2_\epsilon,\gamma}$ bounds a trivial bigon with $\delta_1$.}
    \label{cancelbigon}
\end{figure}

To verify the claim, note that if $B$ is a trivial bigon, i.e., it contains no component of the dividing set in the interior, then we can easily isotop $\beta$ to eliminate $B$. If otherwise, we consider a minimal bigon bounded by $\beta$ and $\delta_1$ in the sense that the interior of the bigon does not intersect with $\beta$. Take a disk $D^2_\epsilon \subset B$ containing all components of the dividing set $\tilde\Gamma$ in $B$, namely, $\Gamma' \cap D^2_\epsilon=\tilde\Gamma$ and $\Gamma'\cap(B \setminus D^2_\epsilon)=\emptyset$. By our assumption, the bypass attachment $\sigma_\beta$ decreases $\#\Gamma'$ by 2, so $\beta$ must intersect $\Gamma'$ in three points which are contained in three different connected components of $\Gamma'$ respectively. One can find an oriented loop $\gamma:[0,1] \to S^2\setminus\Gamma'$ with $\gamma(0)=\gamma(1) \in D^2_\epsilon$ such that $\gamma$ intersects $\beta$ in one point. Orient $\gamma$ in such a way that it goes from $\gamma\cap\beta$ to $\gamma(1)$ in the interior of $B$ as depicted in Figure~\ref{cancelbigon}(b). Suppose that $\Phi:S^2\times[0,1] \to S^2\times[0,1]$ is induced by an isotopy $\phi_t$ which parallel transports $D^2_\epsilon$ along $\gamma$. By pulling up the the bypass attachment $\sigma_\beta$ through $\xi_{\Gamma',\Phi}$, we get the following isotopy of contact structures (cf. proof of Lemma~\ref{SimplePermBCD}):
\begin{align*}
\sigma_\beta \ast \xi_{\Gamma'',\Phi(D^2_\epsilon,\gamma)} \simeq \xi_{\Gamma',\Phi(\tilde\Gamma,D^2_\epsilon,\gamma)} \ast \sigma_{\tilde\beta}
\end{align*}
where $\Gamma''$ is obtained from $\Gamma'$ by attaching a bypass along $\beta$, and $\tilde\beta$ is the pull-up of $\beta$ which is isotopic to the one depicted in Figure~\ref{cancelbigon}(c).

Since $\tilde\beta$ and $\delta_1$ cobound a trivial bigon, a further isotopy of $\tilde\beta$ will eliminate the bigon so that $\beta'$ does not intersect $\delta_1$ in this local picture. By Proposition~\ref{PermEquToBTs}, the contact structure $\xi_{\Gamma',\Phi(\tilde\Gamma,D^2_\epsilon,\gamma)}$ is stably isotopic to $\triangle^n$ for some $n\in\mathbb{N}$. Define $\Phi^{-1}:S^2\times[0,1] \to S^2\times[0,1]$ by $(x,t) \mapsto (\phi^{-1}_t(x),t)$, then it is easy to see that $\xi_{\Gamma'',\Phi(D^2_\epsilon,\gamma)}\ast\xi_{\Gamma'',\Phi^{-1}(D^2_\epsilon,\gamma)}$ is isotopic, relative to the boundary, to an $I$-invariant contact structure. Since we will use this trick many times, we simply write $\xi_{\Phi^{-1}}$ for $\xi_{\Gamma'',\Phi^{-1}(D^2_\epsilon,\gamma)}$ when there is no confusion. To summarize, we have
\begin{align*}
\sigma_\beta &\simeq \xi_{\Gamma',\Phi(\tilde\Gamma,D^2_\epsilon,\gamma)} \ast \sigma_{\tilde\beta} \ast \xi_{\Gamma'',\Phi^{-1}(D^2_\epsilon,\gamma)}\\
             &\sim \triangle^n \ast \sigma_{\tilde\beta} \ast \xi_{\Gamma'',\Phi^{-1}(D^2_\epsilon,\gamma)}\\
             &\simeq \sigma_{\tilde\beta} \ast \triangle^n \ast \xi_{\Gamma'',\Phi^{-1}(D^2_\epsilon,\gamma)}
\end{align*}
By applying the above argument finitely many times, we can eliminate all bigons bounded by $\beta$ and $\delta_1$. Hence the claim is proved.\\

Let us assume that $\beta$ intersects $\delta_1$ nontrivially, and $\beta$ and $\delta_1$ do not cobound any bigon on $S^2$. We consider the following two cases separately.\\

\noindent
\emph{Case 1.} Suppose $\beta$ does not intersect any of the three components of the dividing set generated by the bypass attachment $\sigma_\alpha$. Let $\Gamma_1$, $\Gamma_2$ and $\Gamma_3$ be the three dividing circles which intersect with $\beta$.
\begin{figure}[h]
    \begin{overpic}[scale=.26]{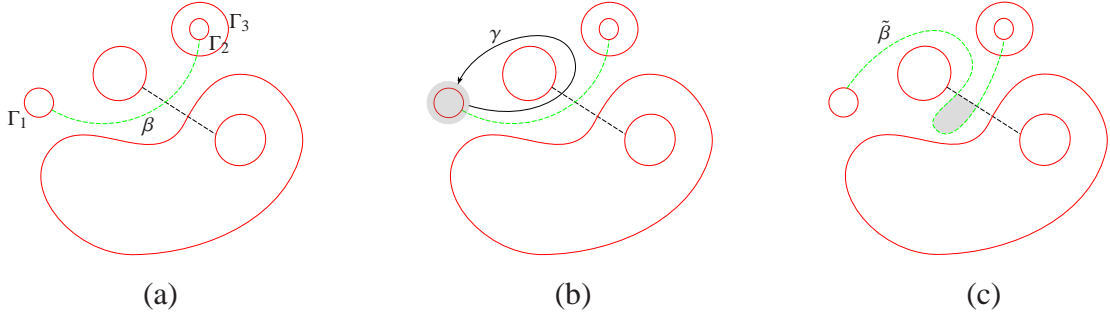}
    \put(-1.5,12){\tiny{$\Gamma_1$}}
    \put(17.1,19){\tiny{$\Gamma_2$}}
    \put(19,21){\tiny{$\Gamma_3$}}
    \put(11,11.5){\tiny{$\beta$}}
    \put(43,20){\tiny{$\gamma$}}
    \put(79,20){\tiny{$\tilde\beta$}}
    \put(11,-4.5){(a)}
    \put(49,-4.5){(b)}
    \put(87,-4.5){(c)}
    \end{overpic}
    \newline
    \caption{(a) The convex sphere $(S^2,\Gamma')$ with an admissible arc $\beta$ intersecting $\delta_1$ in exactly one point. (b) Choose a disk $D^2_\epsilon$ containing $\Gamma_1$ and an oriented loop $\gamma$, along which we apply the isotopy. (c) The pull-up of $\beta$ through the contact structure $\xi_{\Gamma_1,D^2_\epsilon,\gamma}$ bounds a trivial bigon with $\delta_1$.}
    \label{cancelintersection1}
\end{figure}
If $\beta$ intersects $\delta_1$ in exactly one point as depicted in Figure~\ref{cancelintersection1}(a), then we choose a disk $D^2_\epsilon \supset \Gamma_1$ and an oriented loop $\gamma$ in the complement of the dividing set as depicted in Figure~\ref{cancelintersection1}(b) such that $\sigma_\beta \simeq \xi_{\Gamma',\Phi(\Gamma_1,D^2_\epsilon,\gamma)} \ast \sigma_{\tilde\beta} \ast \xi_{\Phi^{-1}} \sim \triangle^m \ast \sigma_{\tilde\beta} \ast \xi_{\Phi^{-1}}$ by arguments as before for some $m\in\mathbb{N}$, where $\tilde\beta$ intersects $\delta^1$ in exactly two points and cobound a trivial bigon as depicted in Figure~\ref{cancelintersection1}(c). Hence an obvious further isotopy of $\tilde\beta$ makes it disjoint from $\delta_1$ as desired.

If $\beta$ intersects $\delta_1$ in more than one point, we orient $\beta$ so that it starts from the point $q=\beta\cap\Gamma_1$ as depicted in Figure~\ref{cancelintersection2}(a). Let $q_1$ and $q_2$ be the first and the second intersection points of $\beta$ with $\delta_1$ respectively. Note that since we assume $\beta$ and $\delta_1$ do not cobound any bigon, there is no more intersection point $\beta\cap\delta_1$ along $\delta_1$ between $q_1$ and $q_2$. Let $\overrightarrow{qq_1}$, $\overrightarrow{q_1q}$ and $\overrightarrow{q_1q_2}$ be oriented subarcs of $\beta$ and $\overrightarrow{q_2q_1}$ be an oriented subarc of $\delta_1$. We obtain a closed, oriented (but not embedded) loop $\gamma=\overrightarrow{qq_1} \cup \overrightarrow{q_1q_2} \cup \overrightarrow{q_2q_1} \cup \overrightarrow{q_1q}$ by gluing the arcs together. To apply Proposition~\ref{PermEquToBTs} in this case, we take an embedded loop close to $\gamma$ as depicted in Figure~\ref{cancelintersection2}(b), which we still denote by $\gamma$. Let $D^2_\epsilon$ be a small disk containing $\Gamma_1$ as usual. Again by pulling up the bypass attachment $\sigma_\beta$ through $\xi_{\Gamma',\Phi(\Gamma_1,D^2_\epsilon,\gamma)}$, we have (stable) isotopies of contact structures $\sigma_\beta \simeq \xi_{\Gamma',\Phi(\Gamma_1,D^2_\epsilon,\gamma)} \ast \sigma_{\tilde\beta}\ast\xi_{\Phi^{-1}} \sim \triangle^r \ast \sigma_{\tilde\beta}\ast\xi_{\Phi^{-1}}$ for some $r\in\mathbb{N}$, where $\tilde\beta$ and $\delta_1$ bound a trivial bigon. Hence an obvious further isotopy eliminates the trivial bigon and decreases $\#(\beta\cap\delta_1)$ by 2. By applying the above argument finitely many times, we can reduce to the case where $\beta$ intersects $\delta_1$ in exactly one point, but we have already solved the problem in this case. We conclude that under the hypothesis at the beginning of this case, there exists a $\tilde\beta$ disjoint with
 $\delta_1$ such that $\sigma_\alpha \ast \sigma_\beta \sim \sigma_\alpha \ast \sigma_{\tilde\beta} \ast \triangle^l \ast \xi_\Phi$ for some isotopy $\Phi$ and an integer $l\in\mathbb{N}$.\\

\begin{figure}[h]
    \begin{overpic}[scale=.22]{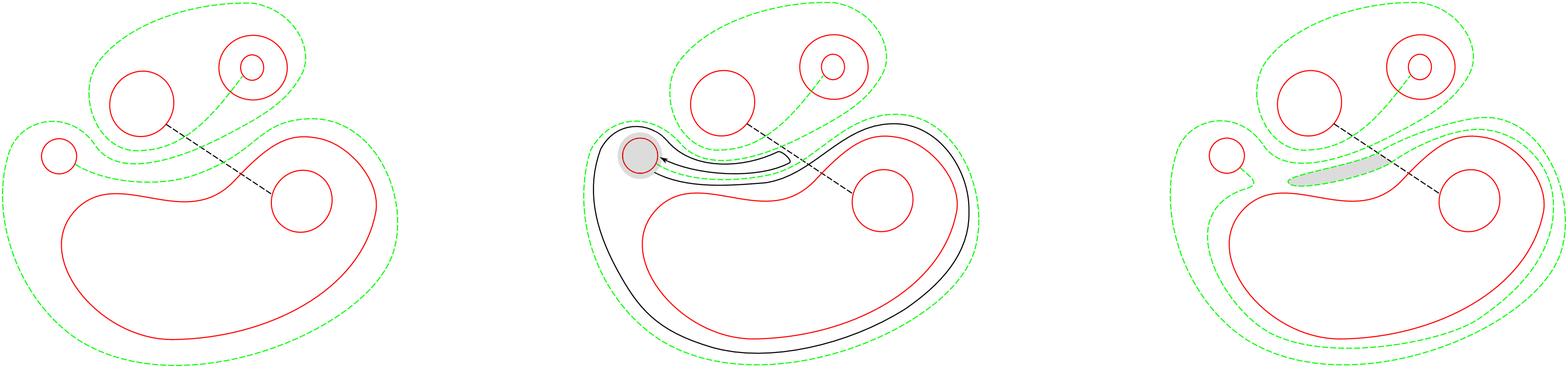}
    \put(1.4,11.5){\tiny{$\Gamma_1$}}
    \put(5,13){\tiny{$q$}}
    \put(13.2,11.7){\tiny{$q_1$}}
    \put(12.8,14.5){\tiny{$q_2$}}
    \put(25,5){\tiny{$\beta$}}
    \put(39.2,7){\tiny{$\gamma$}}
    \put(100,5){\tiny{$\beta'$}}
    \put(10,-4.5){(a)}
    \put(48,-4.5){(b)}
    \put(86,-4.5){(c)}
    \end{overpic}
    \newline
    \caption{(a) The convex sphere $(S^2,\Gamma')$ with an admissible arc $\beta$ intersecting $\delta_1$ in at least two points, say, $q_1$ and $q_2$. (b) The embedded, oriented loop $\gamma$ approximating the broken loop $\vec{qq_1} \cup \vec{q_1q_2} \cup \vec{q_2q_1} \cup \vec{q_1q}$. (c) The pull-up of $\beta$ through the contact structure $\xi_{\Gamma_1,D^2_\epsilon,\gamma}$ bounds a trivial bigon with $\delta_1$.}
    \label{cancelintersection2}
\end{figure}

\begin{figure}[h]
    \begin{overpic}[scale=.25]{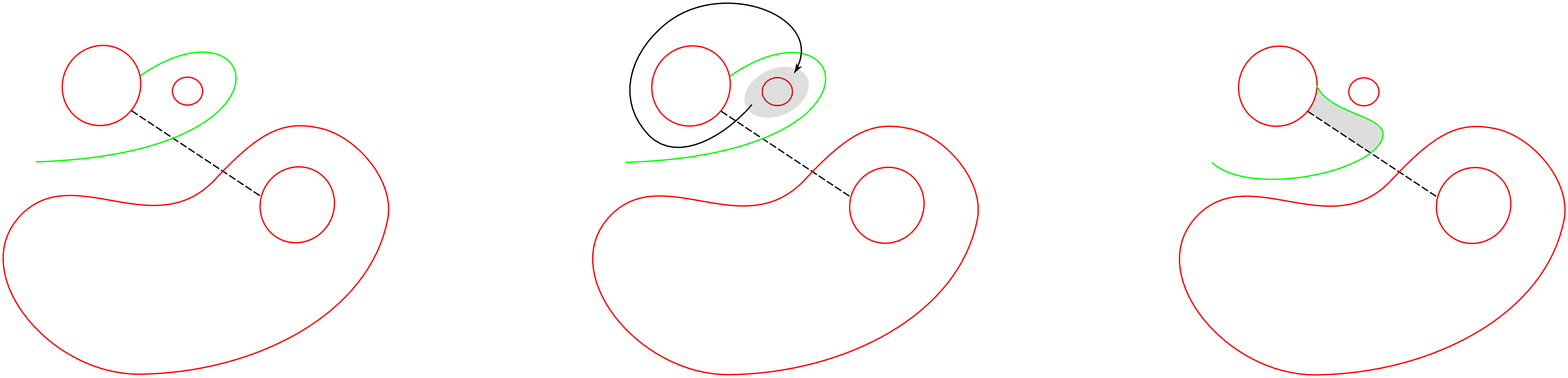}
    \put(8.5,20){\tiny{$r$}}
    \put(6.2,17.5){\tiny{$p_1$}}
    \put(10.3,13.7){\tiny{$r_1$}}
    \put(38.6,18){\tiny{$\gamma$}}
    \put(15.3,18){\tiny{$\beta$}}
    \put(88.7,15){\tiny{$\tilde\beta$}}
    \put(10,-4.5){(a)}
    \put(47,-4.5){(b)}
    \put(86,-4.5){(c)}
    \end{overpic}
    \newline
    \caption{(a) The admissible arc $\beta$, the dividing set $\Gamma'$ and $\delta_1$ cobound a topological triangle $\triangle rr_1p_1$, which may contain other components of the dividing set in the interior. (b) Choose the disk $D^2_\epsilon$ to contain all the components of the dividing set in the topological triangle $\triangle rr_1p_1$, and an oriented loop $\gamma$ which intersects $\beta$ in exactly one point. (c) By applying the isotopy along $\gamma$, the admissible arc $\beta$ becomes $\beta'$ which bounds a trivial triangle with the dividing set and $\delta_1$.}
    \label{cancelintersection3}
\end{figure}

\noindent
\emph{Case 2.} Suppose $\beta$ nontrivially intersects the union of the three components of the dividing set generated by the bypass attachment $\sigma_\alpha$. Without loss of generality, we pick an intersection point $r$ as depicted in Figure~\ref{cancelintersection3}(a). Orient $\beta$ so that it starts from $r$. Let $r_1$ be the first intersection point of $\beta$ and $\delta_1$. Then $\beta$, $\delta_1$ and $\Gamma'$ bound a triangle $\triangle rr_1p_1$. By the assumption that there exists no bigon bounded by $\beta$ and $\delta_1$, the interior of the triangle $\triangle rr_1p_1$ does not intersect with $\beta$. If the interior of the triangle $\triangle rr_1p_1$ contains no components of the dividing set, then it is easy to isotop $\beta$ so that $\#(\beta\cap\delta_1)$ decreases by 1. If otherwise, take a small disk $D^2_\epsilon \subset \triangle rr_1p_1$ containing all components of the dividing set $\tilde\Gamma$ in $\triangle rr_1p_1$, i.e., $\triangle rr_1p_1 \setminus D^2_\epsilon$ does not intersect with the dividing set $\Gamma'$. Let $\gamma$ be an oriented loop based at a point in $D^2_\epsilon$ which does not intersect with the dividing set, and intersects $\beta$ exactly once. By pulling up the bypass attachment $\sigma_\beta$ through $\xi_{\Phi(\tilde\Gamma,D^2_\epsilon,\gamma)}$, we have (stable) isotopies of contact structures $\sigma_\beta \simeq \xi_{\Gamma',\Phi(\tilde\Gamma,D^2_\epsilon,\gamma)}\ast\sigma_{\tilde\beta}\ast\xi_{\Phi^{-1}} \sim \sigma_{\tilde\beta}\ast\triangle^n\ast\xi_{\Phi^{-1}}$ so that $\tilde\beta$, $\delta_1$ and $\Gamma'$ bound a trivial triangle in the sense that the interior of the triangle does not intersect with the dividing set. Hence we can further isotop $\tilde\beta$ to eliminate the trivial triangle and hence decrease $\#(\tilde\beta\cap\delta_1)$ by 1. By applying such isotopies finitely many times, we get an admissible arc $\tilde\beta$ such that $\#(\tilde\beta\cap\delta_1)=0$ and satisfy all the conditions of the proposition.
\end{proof}

\section{Classification of overtwisted contact structures on $S^2\times[0,1]$}

We have established enough techniques to classify overtwisted contact structures on $S^2\times[0,1]$.

\begin{prop} \label{Classifi}
Let $\xi$ be an overtwisted contact structure on $S^2\times[0,1]$ such that $S^2\times\{0,1\}$ is convex with $\Gamma_{S^2\times\{0\}}=\Gamma_{S^2\times\{1\}}=S^1$. Then $\xi\sim\triangle^n$ for some $n\in\mathbb{N}$, where $\triangle^{n}$ denotes the contact structure on $S^2\times[0,1]$ obtained by attaching $n$ bypass triangles to $S^2\times\{0\}$ with the standard tight neighborhood.
\end{prop}

\begin{proof}
By Giroux's criterion of tightness, both $S^2\times\{0\}$ and $S^2\times\{1\}$ have neighborhoods which are tight. Take an increasing sequence $0=t_0<t_1<\cdots<t_n=1$ such that $\xi$ is isotopic to a sequence of bypass attachments $\sigma_{\alpha_0} \ast \sigma_{\alpha_1} \ast\cdots\ast \sigma_{\alpha_{n-1}}$, where $\alpha_i\subset S^2\times\{t_i\}$ are admissible arcs along which a bypass is attached. Define the complexity of a bypass sequence to be $c=\max_{0\leq i \leq n}\#\Gamma_{S^2\times\{t_i\}}$. The idea is to show that if $c>3$, then we can always decrease $c$ by 2 by isotoping the bypass sequence and suitably attaching bypass triangles.

To achieve this goal, we divide the admissible arcs on $(S^2,\Gamma)$ into four types (I), (II), (III) and (IV), according to the number of components of $\Gamma$ intersecting the admissible arc as depicted in Figure~\ref{Types}, where we only draw the dividing set which intersects the admissible arc. Observe that bypass attachment of type (I) increases $\#\Gamma$ by 2, bypass attachment of type (II) and (III) do not change $\#\Gamma$, and bypass attachment of type (IV) decreases $\#\Gamma$ by 2. Hence the complexity of a sequence of bypass attachments changes only if the types of bypasses in the sequence change. By repeated application of Lemma~\ref{DesAscBypass}, we may assume that contact structures induced by isotopies are contained in a neighborhood of $S^2\times\{1\}$. By assumption, $S^2\times\{1\}$ has a tight neighborhood. Hence according to Remark~\ref{uniquetight}, we shall only consider sequences of bypass attachments modulo contact structures induced by isotopies. \\

\begin{figure}[h]
    \begin{overpic}[scale=.26]{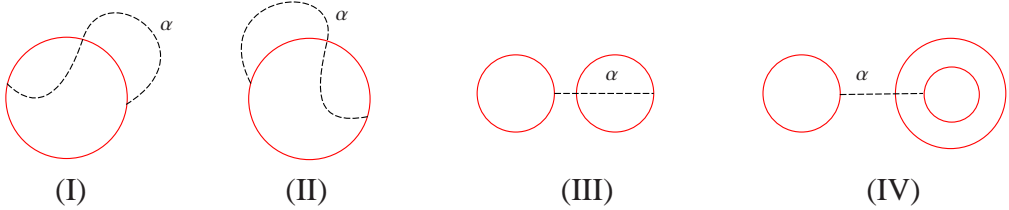}
    \put(15.5,13){\tiny{$\alpha$}}
    \put(33,14){\tiny{$\alpha$}}
    \put(60,8){\tiny{$\alpha$}}
    \put(85,8){\tiny{$\alpha$}}
    \put(5,-4){(I)}
    \put(28,-4){(II)}
    \put(55.5,-4){(III)}
    \put(86,-4){(IV)}
    \end{overpic}
    \newline
    \caption{Four types of admissible arcs $\alpha$ on $(S^2,\Gamma)$.}
    \label{Types}
\end{figure}

\noindent
{\em Claim 1}: We can isotop the sequence of bypass attachments such that only bypasses of type (I) and (IV) appear. \\

To prove the claim, we first show that a bypass attachment of type (III) can be eliminated. Take an admissible arc $\alpha$ of type (III). If the bypass attachment along $\alpha$ is trivial, then by Lemma~\ref{Triviality}, the bypass attachment $\sigma_\alpha$ is induced by an isotopy. Otherwise there exists an admissible arc $\beta$ disjoint from $\alpha$ as depicted in Figure~\ref{ElimTyps}(a)\footnote{In literature, we say $\beta$ is obtained from $\alpha$ by {\em left rotation}.} such that if one attaches a bypass along $\alpha$, followed by a bypass attached along $\beta$, then the later bypass attachment is trivial.

\begin{figure}[h]
    \begin{overpic}[scale=.25]{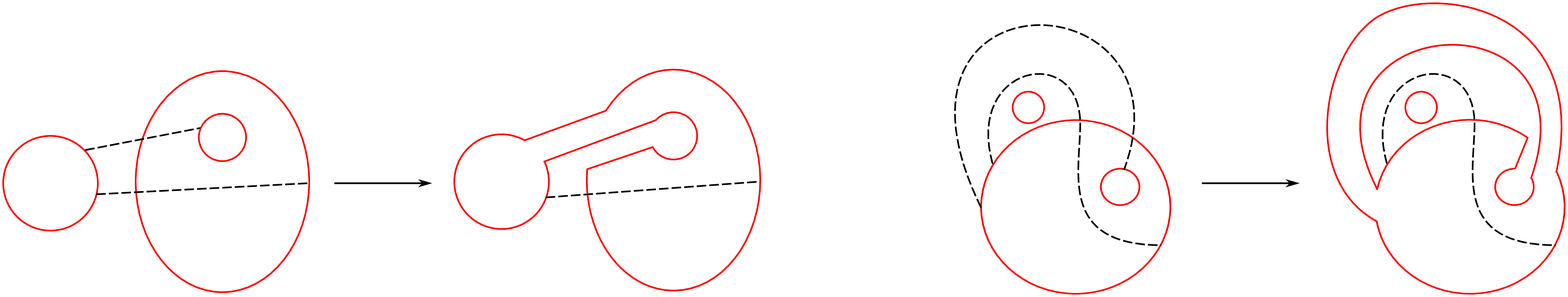}
    \put(12,5){\tiny{$\alpha$}}
    \put(7,11){\tiny{$\beta$}}
    \put(43,5){\tiny{$\alpha$}}
    \put(23,9){\footnotesize{$\sigma_{\beta}$}}
    \put(66.5,5){\tiny{$\alpha$}}
    \put(73,13){\tiny{$\beta$}}
    \put(92,5){\tiny{$\alpha$}}
    \put(78,9){\footnotesize{$\sigma_{\beta}$}}
    \put(23,-4){(a)}
    \put(78,-4){(b)}
    \end{overpic}
    \newline
    \caption{}
    \label{ElimTyps}
\end{figure}

By the disjointness of admissible arcs $\alpha$ and $\beta$, we get the following isotopies of contact structures,
\begin{align*}
\sigma_{\alpha} &\simeq \sigma_{\alpha}\ast\sigma_{\beta} \\
                  &\simeq \sigma_{\beta}\ast\sigma_{\alpha}.
\end{align*}
Observe that $\sigma_{\beta}\ast\sigma_{\alpha}$ is a composition of type (I) and type (IV) bypass attachments. Hence a finite number of such isotopies will eliminate all bypass attachments of type (III) in a sequence.

Similarly suppose that $\sigma_\alpha$ is the bypass attachment of type (II) in a sequence and is nontrivial. Then there must exist other components of the dividing set as shown in Figure~\ref{ElimTyps}(b). Choose an admissible arc $\beta$ disjoint from $\alpha$ as depicted in Figure~\ref{ElimTyps}(b) such that if one attaches a bypass along $\alpha$, followed by a bypass attached along $\beta$, then the later bypass attachment is trivial. By the disjointness of $\alpha$ and $\beta$ again, we get the following isotopies of contact structures:
\begin{align*}
\sigma_{\alpha} &\simeq \sigma_{\alpha}\ast\sigma_{\beta} \\
                  &\simeq \sigma_{\beta}\ast\sigma_{\alpha}.
\end{align*}
Observe that $\sigma_{\beta}\ast\sigma_{\alpha}$ is a composition of bypass attachments both of type (III), hence by a further isotopy will turn $\sigma_\alpha$ into a composition of bypass attachments of type (I) and (IV). A finite number of such isotopies will eliminate bypasses of type (II). The claim follows. \\

From now on, we assume that any bypass attachment in $\sigma_{\alpha_0} \ast \sigma_{\alpha_1} \ast\cdots\ast \sigma_{\alpha_{n-1}}$ either increases or decreases $\#\Gamma$ by 2.

Assume that the complexity of the bypass sequence is achieved at level $S^2\times\{t_r\}$ for some $r\in\{0,1,\cdots,n\}$ and is at least 5, i.e., $\#\Gamma_{S^2\times\{t_r\}}=c\geq5$. Then it is easy to see that $\sigma_{\alpha_{r-1}}$ is type (I) and $\sigma_{\alpha_r}$ is type (IV). By Proposition~\ref{disjointness}, we can always assume that $\alpha_r$ is disjoint from $\alpha_{r-1}$ modulo finitely many bypass triangle attachments. Hence we can view both $\alpha_{r-1}$ and $\alpha_r$ as admissible arcs on $S^2\times\{t_{r-1}\}$. To finish the proof of the proposition, it suffices to prove the following claim. \\

\noindent
{\em Claim 2:} We can isotop the composition of bypass attachments $\sigma_{\alpha_{r-1}}\ast\sigma_{\alpha_r}$ such that the local maximum of $\#\Gamma$ at $S^2\times\{t_r\}$ decreases by at least 2.

To prove the claim, let $\gamma\subset\Gamma_{S^2\times\{t_{r-1}\}}$ be the dividing circle which nontrivially intersects $\alpha_{r-1}$. We do a case-by-case analysis depending on the number of points $\alpha_r$ intersecting with $\gamma$. \\

\noindent {\em Case 1}: If $\alpha_r$ intersects $\gamma$ in at most one point, then one easily check that by applying isotopy $\sigma_{\alpha_{r-1}}\ast\sigma_{\alpha_r} \simeq \sigma_{\alpha_r}\ast\sigma_{\alpha_{r-1}}$ to the sequence of bypass attachments, $\#\Gamma_{S^2\times\{t_r\}}$ decreases by 4. \\

\noindent {\em Case 2}: If $\alpha_r$ intersects $\gamma$ in exactly two points, then once again we apply the isotopy $\sigma_{\alpha_{r-1}}\ast\sigma_{\alpha_r} \simeq \sigma_{\alpha_r}\ast\sigma_{\alpha_{r-1}}$ to the sequence of bypass attachments. Now observe that $\sigma_{\alpha_r}\ast\sigma_{\alpha_{r-1}}$ is a composition of bypass attachments of type (III). In the proof of the claim above, we see that any bypass attachment of type (III) is isotopic to a composition of a bypass attachment of type (IV) followed by a bypass attachment of type (I). Such an isotopy also decreases the local maximum of $\#\Gamma$ by 4. \\

\noindent {\em Case 3}: If $\alpha_r$ also intersects $\gamma$ in three points, we consider a disk $D$ bounded by $\gamma$ and $\alpha_{r-1}$ as depicted in Figure~\ref{Case3}(a). If $D$ contains no component of the dividing set in the interior, then $\sigma_{\alpha_{r-1}}\ast\sigma_{\alpha_r}$ is isotopic to a bypass triangle attachment, more precisely, there exists a trivial bypass along an admissible arc $\delta$ on $S^2\times\{t_r\}$ such that $\sigma_{\alpha_{r-1}}\ast\sigma_{\alpha_r}\ast\sigma_\delta$ is a bypass triangle attachment along $\alpha_{r-1}$. Suppose $D$ contains at least one connected component of the dividing set. Let $\beta$ be an admissible arc on $S^2\times\{t_{r-1}\}$ disjoint from $\alpha_{r-1}$ and $\alpha_r$ such that it intersects $\gamma$ in two points and the dividing set contained in $D$ in one point as depicted in Figure~\ref{Case3}(b). \\

\begin{figure}[h]
    \begin{overpic}[scale=.3]{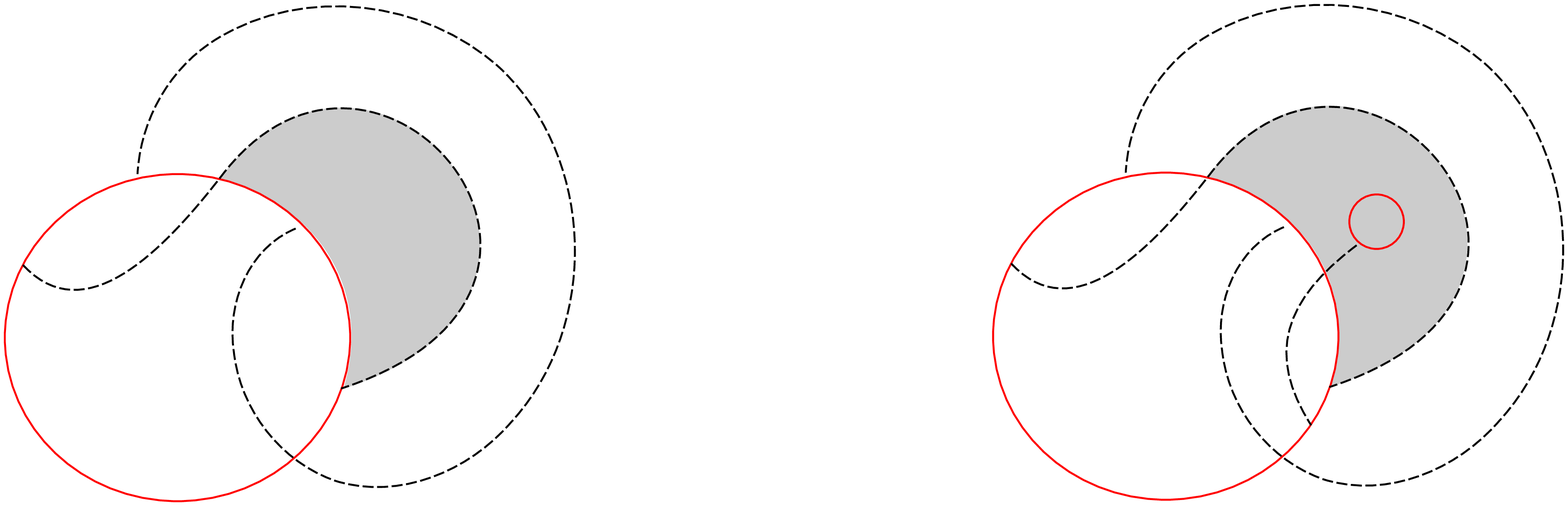}
    \put(29,22){\tiny{$\alpha_{r-1}$}}
    \put(92.5,22){\tiny{$\alpha_{r-1}$}}
    \put(36,24){\tiny{$\alpha_r$}}
    \put(99,24){\tiny{$\alpha_r$}}
    \put(23,17){\small{$D$}}
    \put(0,3){\tiny{$\gamma$}}
    \put(63.5,3){\tiny{$\gamma$}}
    \put(80.5,11){\tiny{$\beta$}}
    \put(17,-5){(a)}
    \put(80,-5){(b)}
    \end{overpic}
    \newline
    \caption{}
    \label{Case3}
\end{figure}

We have the following isotopies of contact structures due to Lemma~\ref{LemBT} and the disjointness of admissible arcs:
\begin{align*}
\sigma_{\alpha_{r-1}}\ast\sigma_{\alpha_r}\ast\triangle
&\simeq\sigma_{\alpha_{r-1}}\ast\sigma_{\alpha_r}\ast\triangle_\beta \\
&=\sigma_{\alpha_{r-1}}\ast\sigma_{\alpha_r}\ast\sigma_\beta\ast\sigma_{\beta'}\ast\sigma_{\beta''} \\
&\simeq\sigma_\beta\ast\sigma_{\alpha_{r-1}}\ast\sigma_{\alpha_r}\ast\sigma_{\beta'}\ast\sigma_{\beta''}
\end{align*}

One can check that the last five bypass attachments above are all of type (III). Hence we can further isotop as before to eliminate type (III) bypass attachments to decrease the local maximum of $\#\Gamma$ by 2.

To summarize, we have proved that any sequence of bypass attachments $\sigma_{\alpha_0} \ast \sigma_{\alpha_1} \ast\cdots\ast \sigma_{\alpha_{n-1}}$ on $S^2\times[0,1]$ is stably isotopic to another sequence of bypass attachments whose complexity is at most 3, which is clearly isotopic to a power of bypass triangle attachments. Thus the proposition is proved.
\end{proof}

\section{Proof of the main theorem}

Now we are ready to finish the proof of Theorem~\ref{OT}.\\

\noindent
{\em Proof of Theorem~\ref{OT}.} By Proposition~\ref{2skeleton}, we can isotop $\xi$ and $\xi'$ so that they agree in a neighborhood of the 2-skeleton. Without loss of generality, we can furthermore assume that there exists an embedded closed ball $B^3\subset M$ such that
\be
\item{$\bdry B^3$ is convex and has a tight neighborhood in $M$ with respect to both $\xi$ and $\xi'$.}
\item{$\xi=\xi'$ in $M\setminus B^3$.}
\item{The restriction of $\xi$ and $\xi'$ to $M\setminus B^3$ and to $B^3$ are all overtwisted.}
\ee

Take a small ball $B_\epsilon^3 \subset B^3$ in a Darboux chart so that both $\xi|_{B_\epsilon^3}$ and $\xi'|_{B_\epsilon^3}$ are tight. We identify $B^3\setminus B_\epsilon^3$ with $S^2\times[0,1]$ and represent the contact structures $\xi|_{B^3\setminus B_\epsilon^3}$ and $\xi'|_{B^3\setminus B_\epsilon^3}$ by two sequences of bypass attachments. By Proposition~\ref{Classifi}, both $\xi|_{B^3\setminus B_\epsilon^3}$ and $\xi'|_{B^3\setminus B_\epsilon^3}$ are stably isotopic to some power of the bypass triangle attachment, in other words, there are isotopies of contact structures $\xi|_{B^3\setminus B_\epsilon^3}\ast\triangle^r \simeq \triangle^{n+r}$ and $\xi'|_{B^3\setminus B_\epsilon^3}\ast\triangle^s \simeq \triangle^{m+s}$ for some $n,m,r,s\in\mathbb{N}$. By assumption, the restriction of $\xi$ and $\xi'$ to $M\setminus B^3$ are overtwisted, so there exist bypass triangle attachments along any admissible arc on $\bdry B^3$ according to Lemma~\ref{Abund}. By simultaneously attaching sufficiently many bypass triangles to $\xi|_{B^3\setminus B_\epsilon^3}$ and $\xi'|_{B^3\setminus B_\epsilon^3}$, we can further assume that $\xi|_{B^3\setminus B_\epsilon^3} \simeq \triangle^n$, $\xi'|_{B^3\setminus B_\epsilon^3} \simeq \triangle^m$ and $\xi=\xi'$ on $M\setminus B^3$.

Let $d$ be the largest integer such that the Euler class $e(\xi)=e(\xi')\in H^2(M;\mathbb{Z})$ divided by $d$ is still an integral class. Such a $d$ is known as the divisibility of the Euler class. Combining Proposition 2.11 and Theorem 0.5 in~\cite{Hu}, we have $d|(m-n)$. To complete the proof of the theorem, we need to show that $\xi|_{M\setminus B^3}$ is isotopic to $\xi|_{M\setminus B^3}\ast\triangle^d$ relative to the boundary. Since $d=g.c.d.\{e(\Sigma)|\Sigma\in H_2(M)\}$, it suffices to prove the following more general fact.

\begin{lemma}
Let $\Sigma$ be a closed surface of genus $g$ and $\eta$ be an $I$-invariant contact structure on $\Sigma\times[0,1]$. Then $\eta\ast\triangle^l$ is stably isotopic to $\eta$ relative to the boundary, where $l=e(\eta)(\Sigma)$.
\end{lemma}

\begin{proof}
Since we only consider stable isotopies of contact structures, one can prescribe any dividing set $\Gamma_\Sigma$ on $\Sigma$ such that the Euler class evaluates on $\Sigma$ to $l$. In particular, we consider the dividing set on $\Sigma$ as depicted in Figure~\ref{Torsion}, namely, there are $g+1$ circles $\gamma_1\cup\cdots\cup\gamma_{g+1}$ dividing $\Sigma$ into two punctured disks, in each of which there are $p$ and $q$ isolated circles respectively. We call the left most circles in the sets of $p$ and $q$ isolated circles $\Gamma_0$ and $\Gamma_1$ respectively. We also choose admissible arcs $\{\alpha_1,\alpha_2,\cdots,\alpha_{p-1}\}$ and $\{\beta_1,\beta_2,\cdots,\beta_{q-1}\}$, and orient $\gamma_i$, $1\leq i\leq g+1$, in a way as depicted in Figure~\ref{Torsion}.

\begin{figure}[h]
    \begin{overpic}[scale=.35]{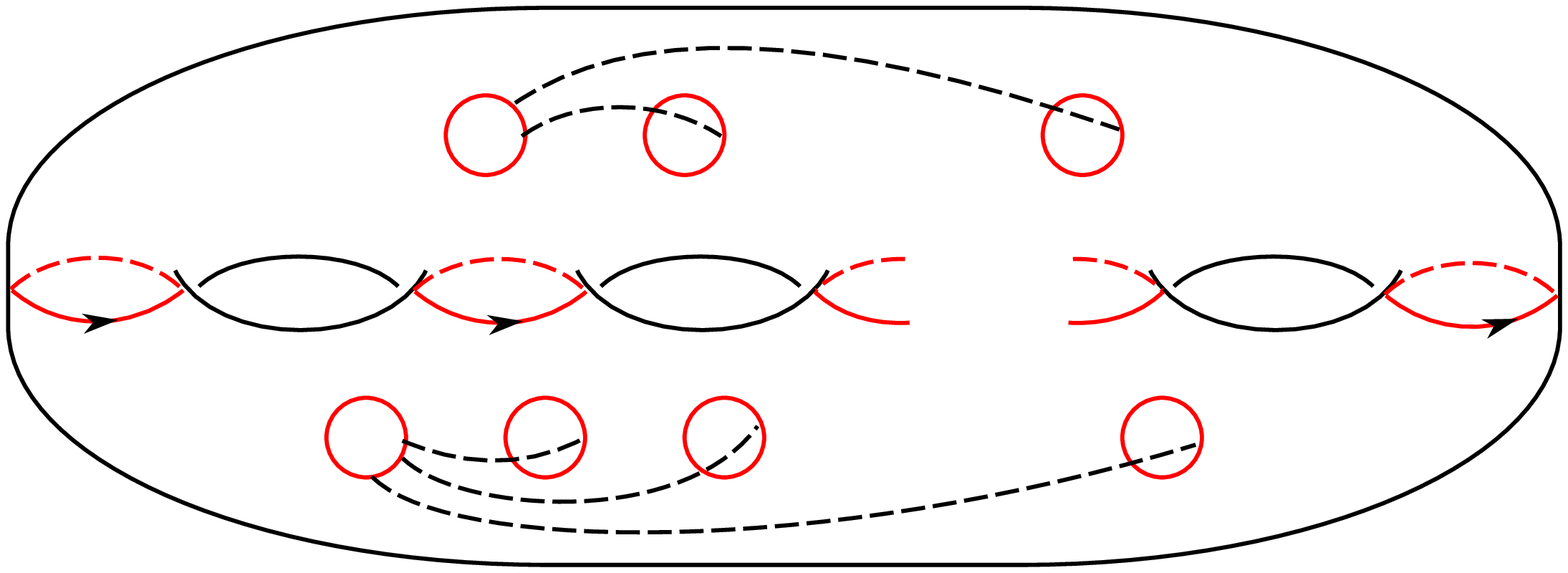}
    \put(8,14){\tiny{$\gamma_1$}}
    \put(34.5,14){\tiny{$\gamma_2$}}
    \put(88.5,13.5){\tiny{$\gamma_{g+1}$}}
    \put(28,8){\tiny{$\alpha_1$}}
    \put(38,6){\tiny{$\alpha_2$}}
    \put(66,3){\tiny{$\alpha_{p-1}$}}
    \put(57,8){$\dots$}
    \put(36.5,26){\tiny{$\beta_1$}}
    \put(62.5,32.2){\tiny{$\beta_{q-1}$}}
    \put(53,27){$\dots$}
    \put(60.5,17.5){$\dots$}
    \put(12,8){\tiny{$-$}}
    \put(12,24){\tiny{$+$}}
    \put(22,7.3){\tiny{$+$}}
    \put(33.3,8.2){\tiny{$+$}}
    \put(44.7,8){\tiny{$+$}}
    \put(73.2,8){\tiny{$+$}}
    \put(29.8,26.5){\tiny{$-$}}
    \put(42.6,26.3){\tiny{$-$}}
    \put(68,26.3){\tiny{$-$}}
    \put(97,30){$\Sigma$}
    \put(17,10.4){\tiny{$\Gamma_0$}}
    \put(25,30.4){\tiny{$\Gamma_1$}}
    \end{overpic}
    \caption{}
    \label{Torsion}
\end{figure}

An easy calculation shows that $l=2(p-q)$. Choose small disks $D^2_{\epsilon,0}$, $D^2_{\epsilon,1}$ in $\Sigma$ such that $D^2_{\epsilon,0} \cap \Gamma_\Sigma=\Gamma_0$ and $D^2_{\epsilon,1} \cap \Gamma_\Sigma=\Gamma_1$. Observe that the bypass triangle attachment along any $\alpha_i$ and $\beta_j$ consists of three trivial bypass attachments, hence is isotopic to contact structures induced by a pure braid of the dividing set. More precisely, let $\gamma^-_i$, $i=1,2,\cdots,g+1$, be an oriented loop in the negative region which is parallel to $\gamma_i$. We have the following isotopies of contact structures $\triangle_{\alpha_1}^2\ast\cdots\ast\triangle_{\alpha_{p-1}}^2 \simeq \eta_{\Phi(\Gamma_0,D^2_{\epsilon,0},\gamma^-_1\cup\cdots\cup\gamma^-_{g+1})} \simeq \eta_{\Phi(\Gamma_0,D^2_{\epsilon,0},\gamma^-_1)}\ast\cdots\ast\eta_{\Phi(\Gamma_0,D^2_{\epsilon,0},\gamma^-_{g+1})}$, where we think of $\gamma^-_1\cup\cdots\cup\gamma^-_{g+1}$ as an oriented loop homologous to the union of the $\gamma_i$'s. Similarly one can study the bypass triangle attachments along the $\beta_j$'s, but with an opposite orientation. Let $\gamma^+_i$ be an oriented loop in the positive region which is parallel to $\gamma_i$ for $1\leq i\leq g+1$. We have the following (stable) isotopies of contact structures $\triangle_{\beta_1}^{-2}\ast\cdots\ast\triangle_{\beta_{q-1}}^{-2} \sim \eta_{\Phi(\Gamma_1,D^2_{\epsilon,1},\gamma^+_1\cup\cdots\cup\gamma^+_{g+1})} \simeq \eta_{\Phi(\Gamma_1,D^2_{\epsilon,1},\gamma^+_1)}\ast\cdots\ast\eta_{\Phi(\Gamma_1,D^2_{\epsilon,1},\gamma^+_{g+1})}$. Here we only have a stable isotopy because of our choice of the orientation of $\gamma_i$. To summarize the computations above, we get the following (stable) isotopies of contact structures:
\begin{align*}
\eta\ast\triangle^l
&\simeq\eta\ast(\triangle_{\alpha_1}^2\ast\cdots\ast\triangle_{\alpha_{p-1}}^2)\ast(\triangle_{\beta_1}^{-2}\ast\cdots\ast\triangle_{\beta_{q-1}}^{-2}) \\
&\simeq\eta\ast(\eta_{\Phi(\Gamma_0,D^2_{\epsilon,0},\gamma_1^-)}\ast\cdots\ast\eta_{\Phi(\Gamma_0,D^2_{\epsilon,0},\gamma_{g+1}^-)})\ast(\eta_{\Phi(\Gamma_1,D^2_{\epsilon,1},\gamma_1^+)}\ast\cdots\ast\eta_{\Phi(\Gamma_1,D^2_{\epsilon,1},\gamma_{g+1}^+)}) \\
&\simeq\eta\ast(\eta_{\Phi(\Gamma_0,D^2_{\epsilon,0},\gamma_1^-)}\ast\eta_{\Phi(\Gamma_1,D^2_{\epsilon,1},\gamma_1^+)})\ast\cdots\ast(\eta_{\Phi(\Gamma_0,D^2_{\epsilon,0},\gamma_{g+1}^-)}\ast\eta_{\Phi(\Gamma_1,D^2_{\epsilon,1},\gamma_{g+1}^+)})
\end{align*}
where the last step follows from the fact that isotopies that parallel transport $D^2_{\epsilon,0}$ and $D^2_{\epsilon,1}$ are disjoint.

Now it suffices to prove that $\eta_{\Phi(\Gamma_0,D^2_{\epsilon,0},\gamma_i^-)}\ast\eta_{\Phi(\Gamma_1,D^2_{\epsilon,1},\gamma_i^+)}$ is stably isotopic to an $I$-invariant contact structure for $1 \leq i\leq g+1$. To see this, take an annular neighborhood $A_i$ of $\gamma_i$ containing $D^2_{\epsilon,0}$ and $D^2_{\epsilon,1}$ and an admissible arc $\delta_i$ which intersects $\Gamma_0$, $\Gamma_1$, and $\gamma_i$ as depicted in Figure~\ref{Annulus}. We can assume that the isotopies $\Phi(\Gamma_0,D^2_{\epsilon,0},\gamma_i^-)$ and $\Phi(\Gamma_1,D^2_{\epsilon,1},\gamma_i^+)$ are supported in $A_i$. For simplicity of notation, we denote the composition $\eta_{\Phi(\Gamma_0,D^2_{\epsilon,0},\gamma_i^-)}\ast\eta_{\Phi(\Gamma_1,D^2_{\epsilon,1},\gamma_i^+)}$ by $\eta_{\gamma_i}$.

\begin{figure}[h]
    \begin{overpic}[scale=.28]{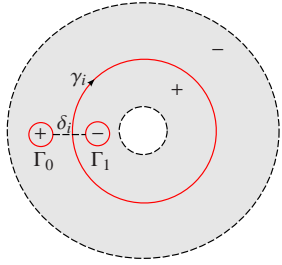}
    \put(10,34){\tiny{$\Gamma_0$}}
    \put(31,34){\tiny{$\Gamma_1$}}
    \put(23.5,65){\tiny{$\gamma_i$}}
    \put(18.43,47.5){\tiny{$\delta_i$}}
    \put(60,60){\tiny{$+$}}
    \put(75,75){\tiny{$-$}}
    \put(10,44){\tiny{$+$}}
    \put(31,44){\tiny{$-$}}
    \end{overpic}
    \caption{An annulus neighborhood $A_i$ of $\gamma_i$ containing $\Gamma_0$ and $\Gamma_1$.}
    \label{Annulus}
\end{figure}

By pushing down the bypass attachment $\sigma_{\delta_i}$ through $\eta_{\gamma_i}$, we have the following isotopies of contact structures:
\begin{align*}
\eta_{\gamma_i}\ast\triangle_{\delta_i}
&=\eta_{\gamma_i}\ast\sigma_{\delta_i}\ast\sigma_{\delta'_i}\ast\sigma_{\delta''_i} \\
&\simeq\sigma_{\tilde\delta_i}\ast\eta_{\Phi(\gamma_i)}\ast\sigma_{\delta'_i}\ast\sigma_{\delta''_i} \\
&\simeq\sigma_{\delta_i}\ast\sigma_{\delta'_i}\ast\sigma_{\delta''_i}=\triangle_{\delta_i}
\end{align*}
where $\tilde\delta_i$ is the push-down of $\delta_i$ which is isotopic to $\delta_i$, and the $\eta_{\Phi(\gamma_i)}$ is easily seen to be isotopic to an $I$-invariant contact structure. The argument works for all $i\in\{1,2,\cdots,g+1\}$, hence we establish the stable isotopy as desired.
\end{proof}

\noindent
{\em Acknowledgements}. The author is very grateful to Ko Honda for inspiring conversations throughout this work. The author also thank MSRI for providing an excellent environment for mathematical research during the academic year 2009-2010.

\end{document}